\documentclass[12pt]{amsart}
\usepackage[utf8]{inputenc}
\usepackage{amsmath, amssymb}
\usepackage{mathtools}
\usepackage{amsthm}
\usepackage{euscript, mathrsfs}
\usepackage{fullpage}
\usepackage{hyperref}
\usepackage{tikz}
\usepackage{xcolor}
\usepackage{array}
\usepackage{url}
\numberwithin{equation}{section}
\usepackage{marginnote}
\usepackage[shortlabels]{enumitem}
\usepackage{appendix}
\usepackage{lscape}
\usepackage{afterpage}
\usepackage{array}
\usepackage{tikz-3dplot}

\usepackage{url}
\usepackage{tkz-euclide}

\renewcommand{\arraystretch}{0.8}
\newcommand*\circled[1]{\tikz[baseline=(char.base)]{
            \node[shape=circle,draw,inner sep=1pt] (char) {#1};}}

\newtheorem{thm}{Theorem}[section]
\newtheorem{cor}[thm]{Corollary}
\newtheorem{lem}[thm]{Lemma}
\newtheorem{prop}[thm]{Proposition}
\theoremstyle{definition}
\newtheorem{ex}[thm]{Example}
\newtheorem{defn}[thm]{Definition}
\newtheorem{rmk}[thm]{Remark}

\def\defcone{\operatorname{Def}}
\def\typecone{\operatorname{TC}}

\def\vert{\operatorname{Vert}}
\def\tes{\operatorname{Tes}}
\def\flow{\operatorname{Flow}}

\def\supp{\operatorname{supp}}

\def\1{{\boldsymbol{1}}}
\def\0{{\boldsymbol{0}}}
\def\ba{{\boldsymbol{a}}}
\def\bb{{\boldsymbol{b}}}
\def\bc{{\boldsymbol{c}}}
\def\be{{\boldsymbol{e}}}

\def\bv{{\boldsymbol{v}}}
\def\bw{{\boldsymbol{w}}}

\def\bg{{\boldsymbol{g}}}
\def\beff{\boldsymbol{f}}
\def\bd{\boldsymbol{d}}

\def\bm{\boldsymbol{m}}

\def\bt{\boldsymbol{t}}

\def\bu{\boldsymbol{u}}
\def\bx{\boldsymbol{x}}

\def\dep{\operatorname{Dep}}

\def\Hs{\boldsymbol{\eta}}
\def\hs{\eta}
\def\balpha{\boldsymbol{\alpha}}
\def\bbeta{\boldsymbol{\beta}}

\def\U{\mathbb{U}}
\def\tU{\widetilde{\mathbb{U}}}

\def\Z{\mathbb{Z}}
\def\R{\mathbb{R}}

\def\cI{\mathcal I}

\makeatletter
\renewcommand*\env@matrix[1][\arraystretch]{%
  \edef\arraystretch{#1}%
  \hskip -\arraycolsep
  \let\@ifnextchar\new@ifnextchar
  \array{*\c@MaxMatrixCols c}}
\makeatother

\newcommand\commentout[1]{}

\title{Deformation cones of Tesler polytopes}
\author{Yonggyu Lee and Fu Liu}
\date{\today}

\address{Yonggyu Lee, Department of Mathematics, University of California, Davis, One Shields Avenue, Davis, CA 95616 USA.}
\email{ygulee@ucdavis.edu}
\address{Fu Liu, Department of Mathematics, University of California, Davis, One Shields Avenue, Davis, CA 95616 USA.}
\email{fuliu@ucdavis.edu}
\begin{document}

\maketitle
\begin{abstract}

For $\ba \in \R_{\geq 0}^{n}$, the Tesler polytope $\tes_{n}(\ba)$ is the set of upper triangular matrices with non-negative entries whose hook sum vector is $\ba$. We first give a different proof of the known fact that for every fixed $\ba_{0} \in \R_{>0}^{n}$, all the Tesler polytopes $\tes_{n}(\ba)$ are deformations of $\tes_{n}(\ba_{0})$. We then calculate the deformation cone of $\tes_{n}(\ba_{0})$. In the process, we also show that any deformation of $\tes_{n}(\ba_{0})$ is a translation of a Tesler polytope. Lastly, we consider a larger family of polytopes called flow polytopes which contains the family of Tesler polytopes and give a characterization on which flow polytopes are deformations of $\tes_{n}(\ba_{0})$.
\end{abstract}

\section{Introduction}

For $\ba \in \R_{\geq 0}^{n}$, the Tesler polytope, denoted as $\tes_{n}(\ba)$, is defined as the set of upper triangular matrices with non-negative entries whose ``hook sum vector'' is $\ba$. When $\ba=\boldsymbol{1}:=(1,\dots,1) \in \R^{n}$, elements of $\tes_{n}(\ba)$ with integer coordinates are called Tesler matrices. Tesler matrices were initially introduced by Tesler in the context of Macdonald polynomials and subsequently rediscovered by Haglund in his work on the diagonal Hilbert series \cite{haglund2011polynomial}. As a result, Tesler matrices play a central role in the field of diagonal harmonics \cite{armstrong2012combinatorics, garsia2014constant, gorsky2015refined, haglund2018delta, wilson2017weighted}. Motivated by this significance of Tesler matrices, M\'{e}sz\'{a}ros, Morales and Rhoades \cite{Meszaros2017} defined and studied the Tesler polytopes of hook sum $\ba \in \Z_{>0}^n$. Their research led to several intriguing findings, including the observation that Tesler polytopes are unimodularly equivalent to certain flow polytopes, which is a family of polytopes that have connection to various areas of mathematics such as toric geometry, representation theory, special functions and algebraic combinatorics, as detailed in \cite{benedetti2019combinatorial} and its references.
Please refer to \cite{Meszaros2017} for more background about Tesler polytopes. 

Beyond their general link to flow polytopes, specific subfamilies of Tesler polytopes are fascinating objects due to their own interesting combinatorial properties and their association with other mathematics domains. We have already mentioned that the importance of the Tesler matrix polytope $\tes_n(\1)$ in the study of diagonal harmonics. Additionally, it is known \cite{corteel2017volumes} that $\tes_{n}(1,0,\dots,0)$ is unimodularly equivalent to the Chan-Robbins-Yuen (CRY) polytope, whose volume is the product of the first $n-2$ Catalan numbers \cite{chan2000volume, zeilberger1999proof}. As of today, simple proofs for this surprising result are still undiscovered. Moreover, the CRY polytope is a face of the Birkhoff polyope which is also a well-studied subject. In particular, computing volumes of Birkhoff polytopes remains a challenging problem that has attracted a lot of recent research attention \cite{canfield2007asymptotic, de2009generating, pak2000four}. 

Noticing that important features remain the same, in our prior work \cite{tesler-pos-bv}, we extend the domain of Tesler polytopes $\tes_n(\ba)$ to $\R_{\geq 0}^{n}$, and study a conjecture regarding Ehrhart positivity for both $\tes_{n}(\boldsymbol{1})$ and $\tes_{n}(1,0,\dots, 0)$, originally posed by Morales \cite{Moralesconjecture}. In this paper, we will continue our study on this broader family of Tesler polytopes. However, our primary focus will be directed towards investigating their deformation cones.

There are many ways of defining deformations of polytopes. The most commonly used one is stated in terms of the normal cone: a polytope $Q$ is a \emph{deformation} of a polytope $P$ if the normal fan of $P$ is a refinement of the normal fan of $Q$. 

However, in this paper we use the one introduced in \cite{castillo2017deformation} by Castillo and the second author; see Definition \ref{defdef} for details. It was shown in \cite[Proposition 2.6]{castillo2017deformation} that Definition \ref{defdef} is equivalent to the normal fan definition of deformations. 
We parametrize each deformation $Q$ of $P$ by a ``deformation vector'' and the collection of the deformation vectors forms the ``deformation cone'' of $P$. 
Notably, the relative interior of the deformation cone is called the type cone \cite{McMullen1973}, which contains the deformation vectors of all the polytopes with the same normal fan as $P$. In the case when $P$ is a rational polytope, the type cone of $P$ is related to the nef cone of the toric variety associated to the normal fan of $P$ \cite[Section 6.3]{cls}.

One of the most famous family of deformations is generalized permutohedra, which are deformations of the regular permutohedron \cite{Edmonds, Pos2009,PRW2008}.  In the last decade, different directions of research have been done on determining the deformation cones of polytopes that are related to generalized permutohedra. In \cite{castillo2017deformation}, Castillo and the second author define the nested permutohedron whose normal fan refines that of the regular permutohedron (so the regular permutohedron is a deformation of the nested permutohedron), and describe the deformation cone of the nested permutohedron. On the other hand, Padrol, Pilaud and Poullot \cite{PPP2023} determine the deformation cones of nestohedra, which is a family of deformations of the regular permutohedron. 

In this paper, we first give a new proof of a known fact (\cite[Remark 9]{baldoni2008kostant}) that for any fixed positive integer $n$ and $\ba_{0}\in \R^{n}_{>0}$, the Tesler polytope $\tes_{n}(\ba)$ is a deformation of $\tes_{n}(\ba_{0})$ for all $\ba \in \R_{\geq 0}^{n}$ (Theorem \ref{main}). 

Then we prove the following theorem - one of our main results - which gives a description of a deformation cone of $\tes_{n}(\ba_{0})$.

(Please refer to Section \ref{preliminary}, especially Definition \ref{defdef} and Section \ref{subsubsec:defthmnota}, for necessary definitions and notations.)

\begin{thm} \label{thm:tesdefcone}
 Let $\ba_{0} \in \R^{n}_{>0}$. Then the deformation cone of $\tes_{n}(\ba_{0})$ with respect to $(L_{n},-P_{n})$ is 
 \begin{equation}\label{eq:tesdefcone}
 \{(\ba,\tilde{\bb})\in \R^{n}\times \tU(n)~~|~~\hs_{i}(\tilde{\bb}) \geq -a_{i} \text{ for all }1 \leq i \leq n-1  \}
 \end{equation}
 where $a_{i}$ is the $i$-th coordinate of $\ba$.
 
 Moreover, any deformation of $\tes_{n}(\ba_{0})$ is a translation of a Tesler polytope. 
\end{thm}
It is worth mentioning that the above description of the deformation cone of $\tes_n(\ba_0)$ is clearly \emph{minimal}, i.e., none of the inequalities is redundant. Therefore, each of the linear inequality defines a facet of the deformation cone. Furthermore, the face structure of the deformation cone of $\tes_n(\ba_0)$ is extremely easy to understand. This will allow us to give a complete characterization on which Tesler polytopes have the same normal fan (Theorem \ref{thm:gen}).

Recall that M\'{e}sz\'{a}ros et al \cite{Meszaros2017} establish that Tesler polytopes is a subfamily of flow polytopes. 
In the last part of this paper, we strengthen this connection by providing the following characterization on which flow polytopes are deformations of $\tes_{n}(\boldsymbol{1})$. (Please see Definitions \ref{defn:flow} and \ref{defn:critical} for the definitions of flow polytopes and critical position.)
\begin{thm} \label{charthm}
  Let $\ba_0 \in \R_{>0}^n$ and let $\ba=(a_{1},\dots,a_{n}) \in \R^{n}$ such that $\flow_n(\ba)$ is non-empty. Suppose $l$ is the critical position of $\ba$. 
Then the flow polytope $\flow_n(\ba)$ is a deformation of $\tes_{n}(\ba_0)$ if and only if $l=n$ or $a_i \ge 0$ for all $l+2 \le i \le n$. 
\end{thm}

\section*{Organization of the paper}
In Section \ref{preliminary}, we provide background on polyhedra theory, deformations of polytopes, and Tesler polytopes. 
In Section \ref{section:deformation}, we give a different proof for the known fact that for every fixed $\ba_0 \in \R_{>0}^n$, all Tesler polytopes $\tes_{n}(\ba)$ are deformations of $\tes_{n}(\ba_0)$, as well as provide a generalization of this fact in Theorem \ref{thm:gen}. In Section \ref{deftrans}, we determine the deformation cone of Tesler polytopes (Theorem \ref{defdef}), and then use it to prove Theorem \ref{thm:gen}. In Section \ref{flowsection}, we discuss which flow polytopes on the complete graph are deformations of $\tes_{n}(\ba_0)$ and complete the proof of Theorem \ref{charthm}.

\subsection*{Acknowledgements} The second author is partially supported by a grant from the Simons Foundation \#426756 and an NSF grant \#2153897-0.

We are grateful to the two anonymous referees for their valuable comments and suggestions. In particular, we are thankful for one of the referees who bring out attention to the result stated in Theorem \ref{thm:gen} which was not in a previous version of this article. 

\section{Preliminaries} \label{preliminary}

In this section, we provide preliminaries and background for this paper. We start with basic definitions for polytopes required for introducing the concepts of deformation cones. 
\subsection{Polytopes} 

A \emph{polyhedron} $P$ is a subset of $\R^{m}$ defined by finitely many linear inequalities \begin{equation}\label{leq} c_{i,1}x_{1}+\cdots+c_{i,k}x_{k} \leq d_{i} \text{ for } 1 \leq i \leq s
\end{equation}
where  $c_{i,j}$ and $d_{i,j}$ are real numbers. We call (\ref{leq}) a \emph{tight} inequality description for $P$ if for every $1 \leq i \leq s$, there exists a point $(x_{i,1},\dots,x_{i,k})$ in $P$ such that $c_{i,1}x_{i,1}+\cdots+c_{i,k}x_{i,k} = d_{i}$. Note that (\ref{leq}) can also be written as $\langle \bc_{i}, \bx \rangle \leq d_{i}$
where $\langle \cdot , \cdot \rangle$ is the dot product on $\R^{m}$.
If we let $C=(c_{i,j})$ and $\bd$ be the column vector $[d_{1},\dots,d_{l}]^{t}$ then the above system can be written in the following matrix form \[C\bx \leq \bd.\] 
A \emph{polytope} is a bounded polyhedron. Equivalently, a polytope $P \subset \R^{m}$ may be defined as the convex hull of finitely many points in $\R^{m}$. A polyhedron defined by homogeneous linear inequalities is called a \emph{(polyhedral) cone}. We assume that readers are familiar with the basic concepts related to polytopes, such as \emph{face} and \emph{dimension}, presented in \cite{barvinok2008integer, ziegler2012lectures}. There are many ways to choose a linear system to define a polytope. 

Let $P_{0} \subset \R^{m}$ be a polytope. We will use the following form to describe $P_{0}$:
\begin{equation} \label{expression}
  E\bx = \balpha_{0} \text{ and } G\bx \leq \bbeta_{0},
\end{equation}
where $E\bx = \balpha_{0}$ defines the affine span of $P_{0}$ and each inequality in $G\bx \leq \bbeta_{0}$ defines a facet of $P_{0}$. 
Any such a description is called a \emph{minimal} linear inequality description for $P_{0}$ if matrices $E$ and $G$ are minimal in size, i.e, none of the equalities or inequalities are redundant. Note that any minimal inequality description is tight. It is not hard to see that minimal inequality descriptions for a non full dimensional polytope $P_{0}$ are not unique. In particular, there are different choices for the matrices $E$ and $G$ in the expression (\ref{expression}). However, any choice of $E$ and $G$ must satisfy that the rows of $E$ are $m-d$ linearly independent vectors and the rows of $A$ are in bijection with the facets of $P_{0}$, where $d$ is the dimension of $P_{0}$. 

\subsection{Deformations}

In this part, we will start with the following definition of deformations initially introduced in \cite{castillo2017deformation}. 
As we mentioned earlier, this definition is equivalent to the normal fan definition of deformations given in the introduction.
Another different but equivalent definition of deformation will then be given in Lemma \ref{combisoeq}.
We finish this part with a discussion on the connection between deformation cones and type cones. 

\begin{defn} \label{defdef}
Let $P_0 \subset \R^{m}$ be a $d$-dimensional polytope, and suppose (\ref{expression}) is a minimal linear inequality description for $P_{0}$. Suppose $G$ has $k$ rows where $\bg_{i}$ is the $i$-th row vector of $G$. For each $i$, let $F_{i}$ be the facet of $P$ defined by $\langle \bg_{i}, \bx \rangle = b_{0,i}$ where $b_{0,i}$ is the $i$-th entry of $\bb_{0}$. 

A polytope $Q \subset \R^{m}$ is a \emph{deformation} of $P_0$, if there exists $\ba \in \R^{m-d}$ and $\bb \in \R^k$ such that the following two conditions are satisfied :
\begin{enumerate}
\item\label{D1} $Q$ is defined by the system of linear inequalities:
\begin{equation}\label{ineqdesc} E\bx = \ba \text{ and }G \bx \leq \bb ,\end{equation}
which is tight (but not necessarily minimal).
\item\label{itm:D2} For any vertex $\bv$ of $P_0$, if $F_{i_1}, F_{i_2}, . . . , F_{i_s}$
are the facets of $P_0$ where $\bv$ lies on, then the
intersection: \[Q \cap \{\bx \in \R^{m} : \langle \bg_{i_j}
, \bx \rangle = b_{i_j},1 \leq j \leq s \} \]
where $b_{i_j}$ is the $i_j$-th component of $\bb$, is a vertex of $Q$ denoted by $\bv_{(\ba, \bb)}$.  

\end{enumerate}
We call $(\ba,\bb)$ the \emph{deforming vector} for $Q$ with respect to ($E$,$G$). Furthermore, if $\bv \mapsto \bv_{(\ba,\bb)}$ gives a bijection from the vertex set of $P_0$ to the vertex set of $Q,$ we say $Q$ is a \emph{weak deformation} of $P_0$; otherwise, we say $Q$ is a \emph{strong deformation} of $P_0.$

The \emph{deformation cone} of $P_{0}$ with respect to ($E, G$), denoted by $\defcone_{(E,G)}({P_{0}})$, is the collection of all the deforming vectors $(\ba,\bb)$ with respect to ($E, G$). 
\end{defn}

Clearly, if we use a different minimal inequality description for $P_{0}$ with matrices $(E',G')$, we will obtain a different deformation cone $\defcone_{(E',G')}(P_{0})$. However, one can check that $\defcone_{(E',G')}(P_{0})$ can be obtained from $\defcone_{(E,G)}(P_{0})$ via an invertible linear transformation. In this paper, we will fix a minimal linear inequality description for the family of polytopes we consider, and thus we will sometimes omit the subscript $(E,G)$ and just write $\defcone(P_{0})$. 

\begin{rmk}\label{rmk:tight}

  With a fixed minimal inequality description for $P_{0}$, the tightness requirement of the description \eqref{ineqdesc} for $Q$ in Definition \ref{defdef} guarantees the uniqueness of the deforming vector $(\ba,\bb)$ for any deformation $Q$ of $P_0$, and establishes a one to one correspondence between deformations $Q$ of $P_{0}$ and deforming vectors $(\ba,\bb)$. 

  In particular, we want to emphasize that if $Q$ is a deformation of $P_0$, its corresponding deforming vector is the \emph{unique} vector $(\ba, \bb)$ such that the description \eqref{ineqdesc} is a \emph{tight} linear inequality description. 
\end{rmk}

\begin{rmk}\label{rmk:translation}
  Recall that a polytope $Q_1 \subset \R^m$ is a \emph{translation} of another polytope $Q_2 \subset \R^m$ if $Q_1 = Q_2 + \bt$ for some $\bt \in \R^m$. It is easy to see that if $Q$ is a deformation of $P$, then any translation of $Q$ is a deformation of $P$.
\end{rmk}

The following lemma provides an alternative definition of deformations that will be used in our proofs. We use $\vert(P)$ to denote the set of vertices of a polytope $P$:

\begin{lem}\cite{postnikov2008faces}\label{combisoeq}
  Let $P_{0}$ and $Q$ be two polytopes in $\R^{m}$. Then the following two statements are equivalent: 
\begin{enumerate}[label={\rm (\arabic*)}]
  \item 
    $Q$ is a deformation of $P_{0}$.
  \item\label{itm:phi} There exists a surjective map $\phi$ from $\vert(P_{0})$ to $\vert(Q)$ satisfying that for any adjacent vertices $\bv$ and $\bw$ of $P_{0}$, there exists $r_{\bv,\bw}\in \R_{\geq 0}$ such that $\phi(\bv)-\phi(\bw)=r_{\bv,\bw}(\bv-\bw)$.
\end{enumerate}

Moreover, suppose \ref{itm:phi} is true, then $Q$ is a {weak deformation} of $P_{0}$ if all $r_{\bv,\bw}$'s are positive and $Q$ is a {strong deformation} of $P_{0}$ if some of $r_{\bv,\bw}$'s are zero. 
\end{lem}
Recall that the direction of a non-zero vector $\bv$ is $\dfrac{\bv}{\sqrt{\langle\bv, \bv \rangle}}$. For any two adjacent vertices $\bv$ and $\bw$ of a polytope, we call the vector $\bw-\bv$ the \emph{edge vector} from $\bv$ to $\bw$, and the direction of $\bw-\bv$ the \emph{edge direction} from $\bv$ to $\bw$. It is not explicitly stated in the above lemma, but when $\bv$ and $\bw$ are adjacent vertices of $P_{0}$, the vertices $\phi(\bv)$ and $\phi(\bw)$ of $Q$ are either the same or adjacent. Moreover, when $\phi(\bv)$ and $\phi(\bw)$ are adjacent, the edge direction of $\phi(\bv)$ and $\phi(\bw)$ is the same as that of $\bv$ and $\bw$.

In \cite[Proposition 2.16]{castillo2017deformation}, the authors gave connection between Definition \ref{defdef} and the alternative definition of deformations provided by Lemma \ref{combisoeq} which we summarize in the following lemma. 
\begin{lem} \cite[Section 2]{castillo2017deformation} \label{philem}
  Let $Q$ be a deformation of $P_{0}$ with deforming vector $(\ba,\bb)$. Then, the map \[\phi:\vert(P_{0}) \rightarrow \vert(Q),\] 
  defined by $\phi(\bv)=\bv_{(\ba,\bb)}$ where $\bv_{(\ba,\bb)}$ is given as in Definition \ref{defdef}/\eqref{itm:D2}  satisfies the condition given in Lemma \ref{combisoeq}/\ref{itm:phi}.
\end{lem}

\subsection*{Connection to type cones} 

  It is straightforward to verify that $Q$ is a weak deformation of $P_0$ if and only if $Q$ has the same normal fan as $P_0$. Thus, in this case, we also say that $Q$ and $P_0$ are \emph{normally equivalent}.
In \cite{McMullen1973}, McMullen introduced the concept of type cones of normally equivalent polytopes. We give its definition using the language developed in this paper.

\begin{defn}\label{defn:type}
  Assuming the hypothesis of $P_0$ given in Definition \ref{defdef}. The \emph{type cone} of $P_0$ with respect to $(E, G),$ denoted by $\typecone_{(E,G)}(P_0)$, is the collection of all the deforming vectors $(\ba,\bb)$ for weak deformations of $P_0$ with respect to ($E, G$). 
\end{defn}

Therefore, the type cone of $P_0$ containing the deformation vectors of all the polytopes that are normally equivalent to $P_0$. Even though McMullen \cite{McMullen1973} did not introduce the terminology ``deformation cone'', his analysis on type cones indicates the following connection between these two families of cones. 

\begin{lem}\label{lem:type}
The relative interior of the deformation cone $\defcone_{(E,G)}(P_0)$ of $P_0$ is exactly the type cone $\typecone_{(E,G)}(P_0)$ of $P_0$.

Furthermore, suppose $K$ is a face of the deformation cone $\defcone_{(E,G)}(P_0)$ of $P_0$, and $Q$ is a deformation of $P_0$ corresponding to a deforming vector in the relative interior of $K$. Then $K$ is exactly the deformation cone $\defcone_{(E,G)}(Q)$ of $Q$ and hence the relative interior of $K$ is the type cone $\typecone_{(E,G)}(Q)$ of $Q$. 
\end{lem}

The following immediate consequence of the above lemma will be useful to us.
\begin{cor}\label{cor:type}
For $i=1,2$, suppose $Q_i$ is a deformation of $P_0$ with the deforming vector $(\ba_i, \bb_i)$, and suppose $K_i$ is the (unique) face of $\defcone_{(E,G)}(P_0)$ such that $(\ba_i, \bb_i)$ is in the relative interior of $K_i$. Then $Q_1$ is a deformation of $Q_2$ if and only if $K_1$ is a face of $K_2$. Moreover, $Q_1$ is normally equivalent to $Q_2$ if and only if $K_1 = K_2$. 
\end{cor}

\subsection{Tesler polytopes} \label{subsec:tesler} 

In this part, we formally introduce Tesler polytopes and provide necessary setup and notations to study the deformation cones of Tesler polytopes. 
We start by setting up the space that Tesler polytopes live in. Recall that the directed complete graph $\vec{K}_{n+1}$ is the graph on the vertex set $[n+1]=\{1,2,\dots,n+1\}$ in which there is an edge from the vertex $i$ to the vertex $j$ for all $1 \leq i < j \leq n+1$. For convenience, we use 
\[E_{n+1}:=\{(i,j)~~|~~1 \leq i < j \leq n+1\}\] 
to represent the edge set of $\vec{K}_{n+1}$. Let $\U(n)=\R^{E_{n+1}}$. We will represent the elements $\bm \in \U(n)$ as upper triangular matrices $(m_{i,j})$ in the following way: we use $m_{i,i}$ to denote the coordinate of $\bm$ corresponding to the edge $(i,n+1)$ for each $1 \leq i \leq n$, and $m_{i,j}$ to denote the coordinate of $\bm$ corresponding to the edge $(i,j)$ for each pair $1 \leq i < j \leq n$. Then for any $1 \leq i \leq n$, we define the \emph{$i$-th hook sum} of $\bm \in \U(n)$ to be \[\hs_{i}(\bm):=m_{i,i}+\sum_{i<j} m_{i,j} -\sum_{j<i}m_{j,i}\] and  the \emph{hook sum vector} of $\bm$ to be $\hs(\bm)=(\hs_{1}(\bm),\dots,\hs_{n}(\bm))$. 
\begin{ex}\label{ex:matrix}
Let $\bm=\begin{bmatrix}
1 & {\color{blue}2} & 3 \\
 & {\color{red}4} & {\color{red}5} \\
 & &10 \\
\end{bmatrix}$. The second hook sum of $\bm$ is $\hs_{2}(\bm)={\color{red}4}+{\color{red}5}-{\color{blue}2}$=7. One can see that the elements involved in $\hs_{2}(\bm)$ forms a hook in $\bm$. We can similarly compute the other two hook sums of $\bm$ and the hook sum vector of $\bm$ is $\hs(\bm)=(6,7,2)$
\end{ex}

For any $\ba \in \R_{\geq 0}^{n}$, we define the \emph{Tesler polytope of hook sum $\ba$} to be
\begin{equation}\label{eq:defntesler}
  \tes_{n}(\ba) := \biggl{\{}\bm \in \U(n)~~\biggl{|}
  \begin{array}{c}
    \Hs(\bm)=\ba,     \\[2mm]
   m_{i,j} \geq 0  \text{ for }1 \le i \le j \le n
  \end{array}  \biggl{\}}.
\end{equation}  
One sees that the condition $m_{n,n} \ge 0$ is implied by the conditions that $m_{i,n} \ge 0$ for each $1 \le i \le n-1$ and that the $n$-th hook sum $a_n$ is non-negative. Hence, the inequality $m_{n,n}\geq0$ is redundant in the above description. Therefore, we use the following inequality description as the definition of $\tes_{n}(\ba)$:
\begin{equation} \label{teslerdef}
    \tes_{n}(\ba) := \biggl{\{}\bm=(m_{i,j}) \in \U(n)~~\biggl{|}\begin{array}{c}
    \Hs(\bm)=\ba,     \\[2mm]
    m_{i,j} \geq 0  \text{ for }1 \le i \le j \le n \text{ where }(i,j) \neq (n,n)
  \end{array}  \biggl{\}}. 
\end{equation}

\subsubsection{Notations for Theorem \ref{thm:tesdefcone}} \label{subsubsec:defthmnota}
In order to describe deformation cones of Tesler polytopes (in Theorem \ref{thm:tesdefcone}), we need to express Tesler polytopes in the form of (\ref{expression}). Note that $\bm \mapsto \Hs(\bm)$ is a linear transformation from $\U(n)$ to $\R^{n}$. Let $L_{n}$ be the matrix associated with this linear transformation. More precisely, we can describe $L_{n}$ in the following way: $L_{n}$ can be obtained from removing the last row of the incidence matrix\footnote{\emph{Incidence matrix} of a directed graph on $n+1$ nodes \{1, 2, \dots, n+1\} is a matrix $\rm{B}$$=(b_{i,e})$ where $b_{i,e}=1$ if the vertex $i$ is the initial vertex of the edge $e$, $b_{i,e}=-1$ if the vertex $i$ is the terminal vertex, and $b_{i,e}=0$ otherwise.} of $\vec{K}_{n+1}$. The columns of $L_{n}$ are indexed by $E_{n+1}$. Hence, the row vectors of $L_{n}$ can be considered to be in $\U(n)$. For any $\bm \in \U(n)$, we consider $L_{n}\bm$ to be the vector whose $i$-th entry is the dot product of the $i$-th row of $L_{n}$ and $\bm$. With this notation, one can check that $\Hs(\bm)=L_{n}\bm$. Similarly, let $\psi$ be the projection that deletes $m_{n,n}$ for all $\bm=(m_{i,j}) \in \U(n)$ and let $P_{n}$ be the matrix associated with the map $\psi$. Then the inequality description in (\ref{teslerdef}) can be written as $\psi(\bm)\geq \0$ or equivalently, $-P_{n}\bm \leq \0$. Combining these together, we obtain a matrix description for $\tes_{n}(\ba)$: 
\[\tes_{n}(\ba)=\{\bm \in \U(n)~~|~~L_{n}\bm=\ba \text{  and  } -P_{n}\bm \leq \boldsymbol{0}\},\]
It was shown in \cite[Lemma 4.2]{tesler-pos-bv} that for all $\ba=(a_{1},\dots,a_{n}) \in \R_{\geq 0}^{n}$ with $a_{1}>0$, we have that $\tes_{n}(\ba)$ is a full dimensional polytope in the subspace of $\U(n)$ defined by $L_{n}\bm=\ba$. Furthermore, by \cite[Corollary 4.6]{tesler-pos-bv} that for all $\ba \in \R_{>0}^n$, we have that each inequality in $-P_{n}\bm\leq \boldsymbol{0}$ defines a facet of $\tes_n(\ba).$ Therefore, for $\ba_{0} \in \R_{>0}^n$, the matrix description
\begin{equation} \label{mintes}
    L_{n}\bm=\ba_{0} \text{ and }-P_{n}\bm\leq \boldsymbol{0}
\end{equation} is a minimal linear inequality description for $\tes_{n}(\ba_{0})$. 
Hence, The deformation cone of $\tes_n(\ba_0)$ studied in this paper and described in Theorem \ref{thm:tesdefcone} is with respect to $(L_n, -P_n)$.

Clearly, each deforming vector $(\ba,\tilde{\bb})$ lives in $\R^{n}\times \tU(n)$, where $\tU(n)$ is the image of $\U(n)$ under the projection $\psi$. 
One sees that the concept of $i$-th hook sum of a matrix in $\tU(n)$ still makes sense as long as $i \neq n$. Therefore, by abusing the notation, for any $1 \leq i \leq n-1$, we define the \emph{$i$-th hook sum} of $\tilde{\bb}=(\tilde{b}_{i,j}) \in \tU(n)$  to be \[\hs_{i}(\tilde{\bb}):=\tilde{b}_{i,i}+\sum_{i<j} \tilde{b}_{i,j} -\sum_{j<i}\tilde{b}_{j,i}.\]

\subsection{Prior work on Tesler polytopes}

In this part, we review definitions and results related to Tesler polytopes given in \cite{Meszaros2017} and \cite{tesler-pos-bv} that are relevant to this paper. Because we treat $\bm \in \U(n)$ as an upper triangular matrix, we will talk about the $i$-th row or $i$-th column of $\bm$.

For any positive integer $n$ and $\ba \in \Z_{\geq 0}^n$, M\'{e}sz\'{a}ros, Morales and Rhoades \cite{Meszaros2017} gave the characterization for the face poset of $\tes_{n}(\ba)$ using the concept of support. 
Their characterization can be easily generalized to any $\ba \in \R_{\geq 0}^{n}$ by the same proof.

\begin{defn}
For any matrix $\bm$ in $\U(n)$, define the \emph{support} of $\bm$, denoted by $\supp(\bm)$, to be the matrix $(s_{i,j}) \in \U(n)$,
\[\text{ where }s_{i,j}=
\begin{cases}
	1, & \text{if the $(i,j)$-th entry of $\bm$ is not zero,} \\
0, & \text{otherwise.} 
\end{cases}\]
Let $\bm$ and $\bm'$ be two matrices in $\U(n)$. We write $\supp(\bm) \leq \supp(\bm')$, if $\supp(\bm)=(a_{i,j})$ and $\supp(\bm')=(b_{i,j})$ satisfying $a_{i,j} \leq b_{i,j}$ for any $1 \leq i\leq j \leq n$. 
\end{defn}

For $1 \leq i \leq j \leq n$, let 
$H_{i,j}^n$ be the hyperplane consisting of $\bm \in \U(n)$ whose $(i,j)$-th entry is $0,$ that is,
\begin{equation} \label{hyperplane}
H_{i,j}^{n}:=\{\bm=(m_{l,k}) \in \U(n)~~|~~m_{i,j}=0\}. 
\end{equation}
We say the intersection $H_{i_{1},j_{1}}^{n} \cap \cdots \cap H_{i_{k},j_{k}}^{n}$ \emph{does not make any zero rows} if the intersection is not contained in $H_{i,i}^{n} \cap H_{i,i+1}^{n} \cap \cdots \cap H_{i,n}^{n}$ for any $1 \leq i \leq n$.
\begin{thm} \cite[Lemma 2.4]{Meszaros2017}\label{Tesler} 
Suppose $\ba \in \R_{\geq 0}^n$ and let $\bv \in \tes_{n}(\ba)$. Then $\bv$ is a vertex of $\tes_{n}(\ba)$ if and only if $\supp(\bv)$ has at most one $1$ on each row. In particular, when $\ba \in \R_{>0}^{n}$, we have that $\bv$ is a vertex in $\tes_{n}(\ba)$ if and only if each row of $\supp(\bv)$ has exactly one $1$.

\end{thm}

\begin{lem}\label{lem:samev}\cite[Lemmas 2.1 and 2.4]{Meszaros2017}
Let $\ba \in \R_{\geq 0}^n$, and $\bv$ be a vertex of $\tes_n(\ba).$ If $\bu \in \tes_n(\ba)$ satisfying $\supp(\bu) \le \supp(\bv)$, then $\bu = \bv.$
\end{lem}

Recall that two vertices of a polytope are said to be \emph{adjacent} if they are connected by an edge. The last two results of \cite{Meszaros2017} we include here are a characterization for adjacent vertices of $\tes_{n}(\ba)$ and a characterization for a specific type of vertex.

\begin{lem} \cite[Theorem 2.7]{Meszaros2017}\label{adj} 
Let $\ba \in \R_{\geq 0}^{n}$. Two vertices $\bv$ and $\bw$ of $\tes_{n}(\ba)$ are adjacent if and only if for every $1 \leq k \leq n$, the $k$-th row of $\supp(\bw)$ can be obtained from the $k$-th row of $\supp(\bv)$ by exactly one of the following operations:
    \begin{enumerate}[label={\rm (O\arabic*)}]
        \item Leaving the $k$-th row of $\supp(\bv)$ unchanged.
        \item Changing the unique $1$ in $k$-th row of $\supp(\bv)$ to $0$.
	\item\label{itm:o3} Changing a single $0$ in the $k$-th row to a $1$ (if $k$-th row of $\supp(\bv)$ is a zero row).
	\item\label{itm:o4} Moving the unique $1$ in the $k$-th row of $\supp(\bv)$ to a different position in the $k$-th row (this operation must take place in precisely one row of $\supp(\bv)$).
    \end{enumerate}
    In particular, when $\ba \in \R_{>0}^n$, two vertices $\bv,\bw$ are adjacent if and only if there exists a unique $k: 1 \le k \le n$ such that $supp(\bw)$ is obtained from $supp(\bv)$ by moving $1$ on the $k$-th row to a different place on the same row.
\end{lem}

\begin{lem} \cite[Lemma 2.4]{Meszaros2017} \label{specialv}
Let $\ba=(a_{1},\dots,a_{n}) \in \R_{\geq 0}^{n}$. The $k$-th row of a vertex $\bv$ of $\tes_{n}(\ba)$ is a zero row if and only if $a_{k}=0$ and the entries of $k$-th column of $\bv$ above the diagonal (excluding the diagonal entry) are all zero.
\end{lem}

\begin{ex} \label{Tesx}
Let $\ba=(2,2,3,4)$, and
\[\bv=\begin{pmatrix}
 0 & 2 & 0 & 0 \\
  & 0 & 0 & 4 \\
  &  &3 & 0 \\
  & &  & 8 \\
  \end{pmatrix}
  \text{ and }
  \bw=\begin{pmatrix}
  0 & 0 & 2 & 0 \\
  & 0 & 0 & 2 \\
  &  &5 & 0 \\
  & &  & 6 \\
\end{pmatrix}.\]
Then
\[\supp(\bv)=
 \begin{pmatrix}
 0 & 1 & 0 & 0 \\
  & 0 & 0 & 1 \\
  &  &1 & 0 \\
  & &  & 1 \\
\end{pmatrix} \text{ and }
\supp(\bw)=
 \begin{pmatrix}
  0 & 0 & 1 & 0 \\
  & 0 & 0 & 1 \\
  &  &1 & 0 \\
  & &  & 1 \\

\end{pmatrix}.\]
Using Theorem \ref{Tesler}, one can easily check that
$\bv$ and $\bw$ are vertices of $\tes_{4}(2,2,3,4)$. Also, they are adjacent vertices, because $\supp(\bw)$
 can be obtained by moving 1 on the first row of
 $\supp(\bv)$
 to the right.
\end{ex}

\section{All Tesler polytopes are deformations of \texorpdfstring{$\tes_{n}(\1)$}{Tes(\1)}
} \label{section:deformation}

The main goal of this section is to prove the following theorem:

\begin{thm} \label{main}
Let $\ba_{0} \in \R_{> 0}^{n}$. Then for any $\ba=(a_{1},\dots,a_{n}) \in \R_{\geq 0}^{n}$, the Tesler polytope $\tes_{n}(\ba)$ is a deformation of
$\tes_{n}(\ba_{0})$. 
More precisely, $\tes_{n}(\ba)$ is a weak deformation of $\tes_{n}(\ba_{0})$ (equivalently, is normally equivalent to $\tes_{n}(\ba_{0})$) if $\ba \in \R_{>0}^{n-1} \times \R_{\ge0}$, and is a strong deformation of $\tes_{n}(\ba_{0})$ if $a_i =0$ for some $1 \leq i \leq n-1$.

In particular, all Tesler polytopes $\tes_n(\ba)$ (for $\ba \in \R_{\ge0}^n$) are deformations of the Tesler polytope $\tes_n(\1)$.
\end{thm}

 Throughout this section, we use the alternative definition of deformations introduced in Lemma \ref{combisoeq}. We will utilize a machinery that was used in the proof of Lemma 2.4 in \cite{Meszaros2017}.

\begin{defn}
Let $\ba \in \R_{\geq 0}^n$ and $\bv=(v_{i,j})$ be a vertex of $\tes_n(\ba)$.
\begin{enumerate}
    \item Define $j_{\bv}:\{0,1,2,\dots,n\} \longrightarrow \{0,1,2,\dots,n\}$ by 
    \[j_{\bv}(k)=\begin{cases}
    0 & \text{if $k=0$ or the $k$-th row of $\bv$ is a zero row}, \\
    l & \text{if $k\neq0$ and $v_{k,l}$ is the unique nonzero entry on the $k$-th row.}
    \end{cases}\] 
    (This is well-defined by Theorem \ref{Tesler}).
  \item  Since $\bv$ is upper triangular, the sequence $\{k, j_{\bv}(k), j_{\bv}^{2}(k),\dots\}$ is strictly increasing until it stabilizes. Assume that $q$ is the smallest integer such that $j_{\bv}^{q-1}(k)=j_{\bv}^{q}(k)$.
    Define \[\dep_{\bv}(k)=\{(k,j_{\bv}(k)),(j_{\bv}(k),j_{\bv}^2(k)),\dots,(j_{\bv}^{q-1}(k),j_{\bv}^{q}(k))\}.\] 

  \item Let $D_{\bv}(k) \in \U(n)$ be the upper triangular $\{0,1\}$-matrix recording the positions appearing in $\dep_{\bv}(k)$. More precisely,
    \[D_{\bv}(k):=(m_{i,j})
\text{ where }  m_{i,j}=\begin{cases}
1, & \text{if }(i,j) \in \dep_{\bv}(k), \\
0, & \text{if }(i,j) \notin \dep_{\bv}(k)
\end{cases}.\]
\end{enumerate}
\end{defn}

The purpose of this definition is to describe the edge vector from a vertex of $\tes_{n}(\ba)$ to one of its adjacent vertices.

\begin{prop}\label{prop:edge}
Let $\ba=(a_{1},\dots,a_{n}) \in \R_{\geq 0}^n$. Suppose $\bv=(v_{i,j})$ and $\bw=(w_{i,j})$ are adjacent vertices of $\tes_n(\ba)$, and suppose the first row that $\supp(\bv)$ and $\supp(\bw)$ differ is the $k$-th row. Then \[\bw-\bv=c(D_{\bw}(k)-D_{\bv}(k))\] where $c \in \R_{>0}$ is the unique non-zero entry in the $k$-th row of $\bv$ (or that of $\bw$). 
\end{prop}

\begin{ex}
  Let $\bv$ and $\bw$ be as in Example \ref{Tesx}. The supports of $\bw$ and $\bv$ differ at the $1$st row already. Hence $k=1.$ We now compute $D_{\bw}(k)- D_{\bv}(k) = D_{\bw}(1)- D_{\bv}(1)$:
  \[ D_{\bw}(1) - D_{\bv}(1) = \begin{pmatrix}
 0 & 0 & 1 & 0 \\
  & 0 & 0 & 0 \\
  & & 1 & 0 \\
  &&& 0
\end{pmatrix} - \begin{pmatrix}
 0 & 1 & 0 & 0 \\
  & 0 & 0 & 1 \\
  & & 0 & 0 \\
  &&& 1
\end{pmatrix} = \begin{pmatrix}
 0 & -1 & 1 & 0 \\
  & 0 & 0 & -1 \\
  & & 1 & 0 \\
  &&& -1
\end{pmatrix}.\]
Next, we compute the edge vector from $\bv$ to $\bw$ directly: 
  \[ \bw -\bv =\begin{pmatrix}
 0 & 0 & 2 & 0 \\
  & 0 & 0 & 2 \\
  &  &5 & 0 \\
  & &  & 6  \end{pmatrix} - \begin{pmatrix}
0 & 2 & 0 & 0 \\
  & 0 & 0 & 4 \\
  &  &3 & 0 \\
  & &  & 8 
\end{pmatrix} = 2 
\begin{pmatrix}
 0 & -1 & 1 & 0 \\
  & 0 & 0 & -1 \\
  & & 1 & 0 \\
  &&& -1
\end{pmatrix}, \]
which is $2 (D_{\bw}(1)- D_{\bv}(1))$, agreeing with the assertion of Proposition \ref{prop:edge}.
\end{ex}

We need two preliminary lemmas before proving Proposition \ref{prop:edge}.

\begin{lem} \label{lem:Dprop}
  Let $\ba \in \R_{\geq 0}^n$ and $\bv=(v_{i,j}) \in \vert(\tes_n(\ba))$. Suppose the $k$-th row of $\bv$ is a non-zero row, and $c$ is the non-zero entry in the $k$-th row. Assume \[\dep_{\bv}(k)=\{(k,j_{\bv}(k)),(j_{\bv}(k),j_{\bv}^2(k)), \dots,(j_{\bv}^{q-1}(k),j_{\bv}^{q}(k))\}.\]

Then $D_{\bv}(k)$ has the following properties:
\begin{enumerate}[label={\rm (\arabic*)}]

  \item\label{itm:bek} $\Hs(D_{\bv}(k))=\be_k$ where $\be_k$ is the $k$-th standard basis vector of $\mathbb{R}^n$.
  \item\label{itm:entrynonnegative} The entries of $\bv -cD_{\bv}(k)$ are all non-negative. 
\end{enumerate}
\end{lem}

\begin{proof}
For convenince, we use $d_{i,j}$ to denote the $(i,j)$-th entry of $D_{\bv}(k)$.

  By convention, we let $j_{\bv}^{0}(k)=k.$ Then it follows from the definition of $\dep_{\bv}(k)$ that 
  \[ 1 \le j_{\bv}^{0}(k) < j_{\bv}^{1}(k) < j_{\bv}^{2}(k) < \cdots < j_{\bv}^{q-2}(k) < j_{\bv}^{q-1}(k) = j_{\bv}^{q}(k) \le n.\]
Therefore, $d_{i,j}$ is $1$ if $(i,j)=(j_{\bv}^{l-1}(k), j_{\bv}^{l}(k))$ for some $l=1,2,\dots, q,$ and is $0$ otherwise. 
  Clearly, if $i \neq j_{\bv}^{l}(k)$ for some $l = 0, 1, \dots, q-1,$ then $\hs_i(D_{\bv}(k)) = 0.$ If $i = j_{\bv}^{l}(k)$ for some $l = 1, 2, \dots, q-1,$ then
  \[ \hs_i(D_{\bv}(k)) = \sum_{j \ge j_{\bv}^{l}(k)} d_{j_{\bv}^{l}(k), j} - \sum_{j < j_{\bv}^{l}(k)} d_{j, j_{\bv}^{l}(k)}= d_{j_{\bv}^{l}(k), j_{\bv}^{l+1}(k)} - d_{j_{\bv}^{l-1}(k), j_{\bv}^{l}(k)} = 1 - 1= 0.\]
  If $i = k = j_{\bv}^{0}(k),$ then
  \[ \hs_i(D_{\bv}(k)) =\hs_k(D_{\bv}(k))  = \sum_{j \ge k} d_{k,j} - \sum_{j < k} d_{j,k} = d_{k, j_{\bv}^{1}(k)} = 1.\]
  Thus, $\Hs(D_{\bv}(k))=\be_k$, i.e. Property \ref{itm:bek} follows. 

  For Property \ref{itm:entrynonnegative}, we first show that 
\begin{equation}
  v_{j_{\bv}^{l-1}(k),j_{\bv}^{l}(k)} \leq v_{j_{\bv}^{l}(k),j_{\bv}^{l+1}(k)}, \quad \text{for any $l=1, 2, \dots, q-1.$}
	\label{eq:compv}
\end{equation}
Since $\bv \in \tes_n(\ba),$ by the definition of $\tes_n(\ba)$, we have that all the entries in $\bv$ are non-negative, and $\hs_{j_{\bv}^{l}(k)}(\bv)$, the $j_{\bv}^l(k)$-th hook sum of $\bv$ is non-negative.
However,
\[ \hs_{j_{\bv}^{l}(k)}(\bv) = \sum_{j \ge j_{\bv}^{l}(k)} v_{j_{\bv}^{l}(k), j} - \sum_{i < j_{\bv}^{l}(k)} v_{i, j_{\bv}^{l}(k)}.\]
Hence, \eqref{eq:compv} follows from the fact that $v_{j_{\bv}^{l}(k),j_{\bv}^{l+1}(k)}$ is the only positive element in the $j_{\bv}^{l}(k)$-th row of $\bv.$ Given \eqref{eq:compv}, we immediately have that
\[ c= v_{k,j_{\bv}(k)} \leq v_{j_{\bv}^{l}(k),j_{\bv}^{l+1}(k)}, \quad \text{for any $l=1,2, \dots, q-1$.}\]
Thus, Property \ref{itm:entrynonnegative} follows.
\end{proof}

\begin{lem}\label{lem:adj}
	Let $\ba \in \R_{\geq 0}^{n}$. Suppose $\bv$ and $\bw$ of $\tes_{n}(\ba)$ are adjacent vertices of $\tes_n(\ba),$ and suppose that the first row in which $\supp(\bv)$ and $\supp(\bw)$ differ is the $k$-th row. Then the followings are true.
	\begin{enumerate}[label={\rm (\arabic*)}]
	  \item\label{itm:rowsame} The first $k-1$ rows of $\bv$ and $\bw$ are the same.
	  \item\label{itm:uniquerow} The $k$-th row is the (unique) row of $\supp(\bw)$ that is obtained from that of $\supp(\bv)$ by operation \ref{itm:o4} of Lemma \ref{adj}. Furthermore, the unique non-zero entries in the $k$-th rows of $\bv$ and $\bw$ are equal to one another. 
	  \item\label{itm:c} Let $c$ be the non-zero entry in the $k$-th row of $\bw.$ For any $k \le i \le j \le n,$ if $w_{i,j} \neq 0$ (or equivalently $w_{i,j} > 0$) but $v_{i,j} =0$, then we have that $(i,j) \in \dep_\bw(k)$ and $w_{i,j} = c.$

	\end{enumerate}
\end{lem}

\begin{proof}
  By Theorem \ref{Tesler}, each row of $\bv$ has at most one non-negative entry. This implies that for any $i,$ the first $i$ rows of $\supp(\bv)$, together with the hook sum condition $\hs(\bv)=\ba,$ determine the first $i$ rows of $\bv.$ The same thing holds for $\bw.$ Therefore, since $\supp(\bv)$ and $\supp(\bw)$ agree at the first $k-1$ rows, we conclude that \ref{itm:rowsame} is true. 
  
Then using the fact that $\hs_k(\bw) = a_k = \hs_k(\bv)$, we see that
	\begin{equation}
\sum_{j=k}^n v_{k,j} = \sum_{j=k}^n w_{k,j}.
		\label{eq:samesum}
	\end{equation}
	Since $\supp(\bv)$ and $\supp(\bw)$ differ at the $k$-th row, the equation \eqref{eq:samesum} implies that the only operation (among the four operations) listed in Lemma \ref{adj} that can be applied to the $k$-th row of $\supp(\bv)$ to obtain that of $\supp(\bw)$ is operation \ref{itm:o4}. 
	Thus, $\bv$ and $\bw$ each has exactly one non-zero entry in their $k$-th row, but in different places. This together with \eqref{eq:samesum} completes our proof for \ref{itm:uniquerow}. 
	
	Finally, we prove \ref{itm:c} by induction on $i.$ Clearly, it holds when $i=k.$ Assume \ref{itm:c} holds for any $i: k \le i < i_0$ for some $i_0\le n,$ and we consider the cases when $i=i_0.$ One observes that the $i$-th row of $\supp(\bw)$ must be obtained from $\supp(\bv)$ by operation \ref{itm:o3} of Lemma \ref{adj}. Hence, the $i$-th row of $\bv$ is a zero row, and $w_{i,j}$ is the unique non-zero entry in the $i$-th row of $\bw.$ Applying Lemma \ref{specialv} to $\bv$, we have that 
\[ a_i = 0 \quad \text{and} \quad v_{1,i} = v_{2,i} = \cdots = v_{i-1,i} = 0.\]
Thus, $\hs_i(\bw)= a_i=0$, which implies that $w_{i,j} = \sum_{l=1}^{i-1} w_{l,i}.$
Since we have already shown \ref{itm:rowsame} holds, we can rewrite this equality as 
\[ w_{i,j} = \sum_{l=k}^{i-1} w_{l,i}.\]
Let $L = \{ k \le l \le i-1 \ | \ w_{l,i} \neq 0\}.$ Because $w_{i,j} \neq 0,$ we have that $L \neq \emptyset.$ For any $l \in L,$ we have $w_{l,i} \neq 0$ and $v_{l,i} = 0,$ and thus by the induction hypothesis $(l,i) \in \dep_\bw(k)$ and $w_{l,i} = c.$ However, it follows from the definition of $\dep_\bw(k)$ that no two pairs in $\dep_\bw(k)$ share the same second coordinates unless one of the pair is in the form of $(x, x).$ But $l < i$ for any $l \in L.$ We conclude that $L$ contains a unique element, say $l_0.$ Then clearly we have that $w_{i,j} = w_{l_0, i} = c$. Furthermore, since $(l_0,i) \in \dep_\bw(k)$, one verifies that $(i,j)\in \dep_\bw(k),$ completing our inductive proof.
\end{proof}

\begin{proof}[Proof of Proposition \ref{prop:edge}]
By Lemma \ref{lem:adj}/\ref{itm:uniquerow}, we must have that the two non-zero entries in the $k$-th rows of $\bv$ and $\bw$ agree, i.e., $v_{k,j_{\bv}(k)}=w_{k,j_{\bw}(k)}$.  Let \[c:=v_{k,j_{\bv}(k)}=w_{k,j_{\bw}(k)}.\] 
	Then it follows from Lemma \ref{lem:Dprop}/\ref{itm:bek} that $c\Hs(D_{\bv}(k))=c\Hs(D_{\bw}(k))=c\be_{k}$. Thus, if we let 
\[ \bu:=\bw+cD_{\bv}(k)-cD_{\bw}(k),\] then $\Hs(\bu) = \Hs(\bw) =\ba.$
Also, by Lemma \ref{lem:Dprop}/\ref{itm:entrynonnegative}, the entries of $\bw-cD_{\bw}(k)$ are non-negative which implies that the entries of $\bu=\left(\bw-cD_{\bw}(k)\right) +cD_{\bv}(k)$ are non-negative as well. Therefore, we conclude that $\bu \in \tes_{n}(\ba)$. 
It is left to show that $\bu = \bv,$ which by Lemma \ref{lem:samev} can be reduced to showing $\supp(\bu) \le \supp(\bv)$. Since $D_{\bv}(k)$ and $\bw -cD_{\bw}(k)$ have non-negative entries, one sees that showing $\supp(\bu) \leq \supp(\bv)$ is equivalent to showing the following two statements: 
\begin{enumerate}[label=(\roman*)]
	\item $\supp(cD_{\bv}(k)) \leq \supp(\bv)$, and
	\item $\supp(\bw - cD_\bw(k)) \le \supp(\bv).$
\end{enumerate}

We see that (i) follows directly from the definition of $D_{\bv}(k)$.

By Lemma \ref{lem:adj}/\ref{itm:rowsame} and the definition of $D_\bw(k)$, one sees that the first $k-1$ rows of $\bw-cD_{\bw}(k)$ are the same as that of $\bv,$ and the $k$-th row of $\bw - c D_\bw(k)$ is a zero row. One sees that in order to finish proving (ii), it is sufficient to show that for any $k < i \le j \le n,$ if $w_{i,j} \neq 0,$ then either $v_{i,j} \neq 0$ or the $(i,j)$-th entry of $\bw - cD_\bw(k)$ is $0.$
However, if $w_{i,j} \neq 0$ and $v_{i,j} = 0,$ it follows from Lemma \ref{lem:adj}/\ref{itm:c} that $w_{i,j}=c$, and thus the $(i,j)$-th entry of $\bw-cD_\bw(k)$ is $0,$ completing the proof.
\end{proof}

We need one more lemma before proving the result of this section.

\begin{lem} \label{welldef}
  Let $\ba_{0} \in \R_{>0}^{n}$ and $\ba \in \R_{\geq 0}^{n}$. Then for any $\bv \in \vert(\tes_{n}(\ba_{0}))$, there exists a unique $\bv' \in \tes_{n}(\ba)$ such that $\supp(\bv') \leq \supp(\bv)$. Furthermore, this unique $\bv'$ is a vertex of $\tes_n(\ba).$
\end{lem}
\begin{proof}
  Because $\ba_{0} \in \R_{>0}^{n}$,  by Theorem \ref{Tesler}, each row of $\supp(\bv)$ has exactly one $1$. Let 
  $S := \{ (i,j_i) \ | \ 1 \le i \le n\}$
  be the set of indices of the entries in $\supp(\bv)$ that are $1.$
  It is easy to see that there exists a unique solution $\bv' = (v'_{i,j}) \in \U(n)$ satisfying
  \[ \Hs(\bv') =\ba \quad \text{ and } \quad v'_{i,j} = 0 \ \text{for every } (i,j) \not\in S.\]
  Note that the condition $v'_{i,j} = 0$ for every $(i,j) \not\in S$ is equivalent to the condition $\supp(\bv') \le \supp(\bv).$ Hence, the first assertion of the lemma follows.

  Since $\supp(\bv') \leq \supp(\bv)$ implies that there is at most one non-zero entry in each row of $\bv'$, by Theorem \ref{Tesler}, $\bv'$ has to be a vertex of $\tes_{n}(\ba)$.
\end{proof}

Now, we are ready to prove our main result of this section.

\begin{proof}[Proof of Theorem \ref{main}]

We prove the theorem using Lemma \ref{combisoeq}. For any vertex $\bv$ of $\tes_n(\ba_0)$, we let $\phi(\bv)$ be the unique point/vertex $\bv'$ in $\tes_n(\ba)$ satisfying $\supp(\bv') \leq \supp(\bv)$ asserted by Lemma \ref{welldef}. Hence, $\phi$ is a well-defined map from $\vert(\tes_n(\ba_{0}))$ to $\vert(\tes_n(\ba))$. 

For any vertex $\bv'$ of $\tes_n(\ba)$, it follows from Theorem \ref{Tesler} that there exists a $\{0,1\}$-matrix $\bm \in \U(n)$ such that each row of $\bm$ has exactly one $1$ and $\supp(\bv') \le \bm.$ Applying Theorem \ref{Tesler} again, there exists a vertex $\bv$ of $\tes_n(\ba_0)$ such that $\supp(\bv)= \bm.$ Therefore, the surjectivity of $\phi$ follows.

Assume that $\bv$ and $\bw$ are adjacent vertices of $\tes_n(\ba_{0})$ where their support only differ in the $k$-th row. Let $\bv'=\phi(\bv)$ and $\bw'=\phi(\bw)$. In order to apply Lemma \ref{combisoeq}, we need to show that there exists $r_{\bv,\bw} \in \R_{\geq 0}$ such that
\[\bw'-\bv'=r_{\bw,\bv}(\bw-\bv).\]
We consider two cases below.
\begin{enumerate}[label={\bf Case \arabic*}]
  \item\label{itm:case1}
    Suppose the $k$-th row of $\supp(\bv')$ or that of $\supp(\bw')$ is a zero row. Without loss of generality, assume that the $k$-th row of $\supp(\bv')$ is a zero row. Then since $\supp(\bv)$ and $\supp(\bw)$ only differs in the $k$-th row, $\supp(\bv') \leq \supp(\bw)$. Thus, it follows from Lemma \ref{welldef} and the constrution of $\phi$ that $\bv'=\bw'$. Therefore, we can choose $r_{\bv,\bw}=0$.

\item
  Suppose neither the $k$-th row of $\supp(\bv')$ nor that of $\supp(\bw')$ is a zero row. 

  Let $\bv'=(v'_{i,j})$, and $v'_{k,l}$ be the (unique) non-zero entry in the $k$-th row of $\supp(\bv').$ Because $\supp(\bv') \leq \supp(\bv)$, clearly we have $j_{\bv}(k)=l=j_{\bv'}(k)$.
  Next, because $v'_{k,l} \neq 0,$ it follows from Lemma \ref{specialv} that the $l$-th row of $\bv'$ is a non-zero row. Using the same argument as above, we get $j_{\bv}(l) = j_{\bv'}(l),$ or equivalently, $j_{\bv'}^{2}(k)=j_{\bv}^{2}(k)$.
  By iterating this process, we obtain $j_{\bv'}^{s}(k)=j_{\bv}^{s}(k)$ for any integer $s \ge 0$. Therefore, $\dep_{\bv'}(k)=\dep_{\bv}(k)$ or equivalently, $D_{\bv'}(k)=D_{\bv}(k)$. 

  Similarly, we obtain $D_{\bw'}(k)=D_{\bw}(k)$. By Proposition \ref{prop:edge}, 
\[\bw -\bv=c(D_{\bw}(k)-D_{\bv}(k))\text{ and }\bw' -\bv'=d(D_{\bw'}(k)-D_{\bv'}(k))\] where $c$ and $d$ are the unique non-zero (positive) entries of $\bv$ and $\bv'$ respectively. Therefore, we can take $r_{\bw, \bv}=\dfrac{d}{c}$.
\end{enumerate}

In particular, when $\ba \in \R_{>0}^{n-1} \times \R_{\ge 0}$, $r_{\bw,\bv}$ is always positive because \ref{itm:case1} does not happen. However, when some of $a_{i}$ where $1 \leq i \leq n-1$ are zero, one can easily see that \ref{itm:case1} does happen which means the corresponding $r_{\bw,\bv}$'s are zero. This completes our proof. 
\end{proof}

We have proved Theorem \ref{main} using the alternative definition of deformations given in Lemma \ref{combisoeq}. Below we will make a connection to the language used in Definition \ref{defdef}, i.e., the original definition of deformations we introduced. Recall that for $\ba_0 \in \R_{>0}^n,$ the matrix description \eqref{mintes} is a minimal linear description for $\tes_n(\ba_0).$ Therefore, it is natural to ask for $\ba \in \R_{\ge 0}^n$, what the deforming vector for $\tes_n(\ba)$ is. 

\begin{lem}\label{lem:tight}
  Suppose $\ba_{0} \in \R_{>0}^{n}$ and $\ba \in \R_{\geq 0}^{n}$. Then the system of linear inequalities 
  \begin{equation}\label{eq:tes_n(a)}
  L_{n}\bm=\ba \text{ and }-P_{n}\bm\leq \boldsymbol{0}
\end{equation}
  is a tight linear inequality description for the non-empty Tesler polytope $\tes_n(\ba).$ 
  Hence, $\tes_n(\ba)$ is a deformation of $\tes_n(\ba_0)$ with deforming vector $(\ba, \0)$ with respect to $(L_n, -P_n)$.
\end{lem}

\begin{proof} 
We already know that the system of linear equalities \eqref{eq:tes_n(a)} defines the non-empty polytope $\tes_n(\ba).$ Hence, we only need to show the tightness part of the first assertion of the lemma. Note that $-P_{n}\bm\leq \0$ is the same as 
  \[  m_{i,j} \geq 0  \text{ for }1 \le i \le j \le n \text{ where }(i,j) \neq (n,n).\]
  So we need to show that for every $1 \le i \le j \le n$ with $(i,j) \neq (n,n)$, there exists a point $\bm \in \tes_n(\ba)$ such that $m_{i,j}=0.$ We construct two points ${\bm}^1=(m_{i,j}^1)$ and ${\bm}^2=(m_{i,j}^2)$ where
\[  m^1_{i,j}=
  \begin{cases} a_{i}, & \text{if $i=j$}, \\
      0, & \text{otherwise};
    \end{cases} \quad \text{ and } \quad
   m^2_{i,j}=
  \begin{cases} \sum_{l=1}^{i}a_{l}, & \text{if $j=i+1$,} \\
    \sum_{l=1}^{n}a_{l}, & \text{if $(i,j)=(n,n)$}, \\
0 &\text{otherwise}.\end{cases} \]
It is straightforward to verify that $\bm^1, \bm^2 \in \tes_n(\ba)$. By definition, $m^1_{i,j}=0$ for all $1 \le i < j \le n$ and $m^2_{i,i}=0$ for all $1 \le i \le n-1.$ Therefore, the first assertion follows.

Finally, by Theorem \ref{main}, we already know that $\tes_n(\ba)$ is a deformation of $\tes_n(\ba_0).$ Thus, the second assertion follows from the first assertion and Remark \ref{rmk:tight}.
\end{proof}

We finish this section by stating a generalization of Theorem \ref{main}, in response to a question asked by an anonymous referee. 
For each $\ba \in \R_{\ge 0}^n$, we define $\cI(\ba) = \{ i \in [n-1] \ | \ a_i > 0\}$ to be the set of indices in $[n-1]$ such that $a_i$ is nonzero. Note that we do not care about whether $a_n$ is zero or not. 

\begin{thm}\label{thm:gen}
  Let $\ba, \bb \in \R_{\ge 0}^n$. Then $\tes_n(\ba)$ is a deformation of $\tes_n(\bb)$ if $\cI(\ba) \subseteq \cI(\bb)$. Moreover, $\tes_n(\ba)$ is normally equivalent to $\tes_n(\bb)$ if $\cI(\ba) = \cI(\bb)$.
\end{thm} 

One sees that Theorem \ref{main} is a special case of Theorem \ref{thm:gen} when taking $\bb \in \R_{>0}^n$. It is possible to revise our proof for Theorem \ref{main} to give a proof for Theorem \ref{thm:gen}. However, the arguments will be much more complicated. Instead, we will prove Theorem \ref{thm:gen} in the next section as a consequence of Theorem \ref{thm:tesdefcone} and Corollary \ref{cor:type}. 

\section{The deformation cone of Tesler polytopes}
\label{deftrans}

In this section, we give a proof for Theorem \ref{thm:tesdefcone}, which describes the deformation cone of $\tes_{n}(\ba_{0})$ for $\ba_{0} \in \R^{n}_{>0}$, and then prove Theorem \ref{thm:gen} by applying Corollary \ref{cor:type}. In order to prove Theorem \ref{thm:tesdefcone}, we need the following preliminary lemma. Recall that $\tU(n)$ is the image of $\U(n)$ under the projection which deletes the $(n,n)$-th entry. 
For convenience, for any $(\ba,\tilde{\bb}) \in \R^{n}\times \tU(n)$, where $\ba=(a_{1},\dots,a_{n})$ and $\tilde{\bb}=(\tilde{b}_{i,j})$, let 
\begin{equation}
  \label{eq:Qdef}
Q(\ba,\tilde{\bb}):=\{\bm \in \U(n)~~|~~L_{n}\bm=\ba \text{ and }-P_{n}\bm \leq \tilde{\bb}\}.
\end{equation}
Note that for arbitrary choices of $(\ba,\tilde{\bb})$, the above linear inequality description might not be tight. 
\begin{lem} \label{vertexcorres}
  Let $\ba_{0} \in \R_{> 0}^{n}$. Suppose $Q$ is a deformation of $\tes_n(\ba_0)$ with deforming vector $(\ba, \tilde{\bb})$. (Thus, $Q = Q(\ba,\tilde{\bb})$ and $(\ba, \tilde{\bb}) \in \defcone(\tes_{n}(\ba_0))$.)
  Then the following statements are true:
\begin{enumerate}[label={\rm (\arabic*)}]
  \item \label{itm:uniquev} For any $\bv \in \vert(\tes_{n}(\ba_0))$, there exists a unique $\bv'=(v'_{i,j}) \in \vert(Q(\ba,\tilde{\bb}))$ such that for all $1 \leq i \leq j \leq n$ where $(i,j) \neq (n,n)$ if $v_{i,j}=0$ then $v'_{i,j}=-\tilde{b}_{i,j}$ . 

  \item Let $\phi$ be a map from $\vert(\tes_{n}(\ba_0))$ to $\vert(Q(\ba,\tilde{\bb}))$ defined by $\phi(\bv)=\bv'$ where $\bv'$ is the unique vertex assumed by part \ref{itm:uniquev} above. Then for any pair of adjacent vertices $\bv_{1}$ and $\bv_{2}$ of $\tes_{n}(\ba_0)$, there exists a non-negative real number $r$ such that $\phi(\bv_{1})-\phi(\bv_{2})=r(\bv_{1}-\bv_{2})$, i.e, $\phi$ satisfies the condition described in Lemma \ref{combisoeq}/\ref{itm:phi}.
\end{enumerate}
\end{lem}

\begin{proof}
  Let $\bv$ be a vertex of $\tes_n(\ba_0).$ Recall that (\ref{mintes}) is a minimal linear inequality description for $\tes_{n}(\ba_{0})$ which means that each inequality in $-P_{n}\bm \leq \boldsymbol{0}$ is facet-defining. Hence, the set of facets of $\tes_{n}(\ba_{0})$ containing $\bv$ is \[\{H_{i,j}^{n} \cap \tes_{n}(\ba_{0}) ~~|~~v_{i,j}=0\}.\] By Definition \ref{defdef}/\eqref{itm:D2}, the vertex $\bv_{(\ba,\tilde{\bb})}$ of $Q(\ba,\tilde{\bb})$ is the unique point $\bv'$ satisfying the desired condition in \ref{itm:uniquev}. Then, the second part follows from Lemma \ref{philem}. 
\end{proof}

We describe a procedure of obtaining $\bv'=\phi(\bv)$ from $\bv$ in the following example.

\begin{ex} Let $\ba_0 = (1, 1, 1, 1)$, $\ba=(8,7,8,1),$ and 
  \[ \tilde{\bb}:=\begin{bmatrix}
-1 & 2 & -3 & -4 \\ 
  & -5 & 6 & 7 \\
  &   & -8 & 9 
\end{bmatrix}.\]
By using Theorem \ref{thm:tesdefcone}, one can check that $Q=Q(\ba, \tilde{\bb})$ is a deformation of $\tes_n(\1)$ with deforming vector $(\ba, \tilde{\bb})$. 
Let \[\bv:=\begin{bmatrix}
0 & 1 & 0 & 0 \\
  & 0 & 2 & 0 \\
  &   & 3 & 0 \\
  &   &   & 1 \\
\end{bmatrix}\]
which is a vertex of $\tes_{n}(\boldsymbol{1})$. By the proof of Lemma \ref{vertexcorres}, we know that $\bv'=\phi(\bv)$ is defined to be $\bv_{(\ba,\tilde{\bb})}$. Hence, we can find $\bv'$ in the following way: We first change zero entries in $\bv$ to corresponding $-\tilde{b}_{i,j}$'s and change the four non-zero entries to undetermined variables $w,x,y,z$:
\[\bv=\begin{bmatrix}
0 & 1 & 0 & 0 \\
  & 0 & 2 & 0 \\
  &   & 3 & 0 \\
  &   &   & 1 
\end{bmatrix} \quad \rightarrow \quad
\bv':=\begin{bmatrix}
1 & w  & 3 & 4 \\
        & 5 & x  &  -7 \\
        &         &  y & -9  \\
        &        &   &  z        
\end{bmatrix}.\]
We then use the hook sum conditions for $Q$ to determine $w,x,y,z$:
\[\begin{cases}
\hs_{1}(\bv')=1+w+3+4=8 \\
\hs_{2}(\bv')=5+x+(-7)-w=7 \\
\hs_{3}(\bv')=y+(-9)-(x+3)=8\\
\hs_{4}(\bv')=z-( (-9) + (-7) +4 )=1
\end{cases} 
\ \ \implies \ \
\begin{cases}
w=0\\
x=9\\
y=29\\
z=-11\\
\end{cases} 
\ \ \implies \ \
\bv'=\begin{bmatrix}
1 & 0  & 3 & 4 \\
        & 5 & 9  &  -7 \\
        &         &  29 & -9  \\
        &        &   &  -11        
\end{bmatrix}.\]
\end{ex}

\begin{proof}[Proof of Theorem \ref{thm:tesdefcone}]
Thanks to Theorem \ref{main}, without loss of generality, we can assume that $\ba_{0}=\boldsymbol{1}$. Let 
\[ A:=\{(\ba,\tilde{\bb})\in \R^{n}\times \tU(n))~~|~~\hs_{i}(\tilde{\bb}) \geq -a_{i} \text{ for all }1 \leq i \leq n-1  \}.\]
We will show that $\defcone(\tes_{n}(\boldsymbol{1}))=A$ by showing two-sided inclusion.

We first show that $\defcone(\tes_{n}(\boldsymbol{1})) \subseteq A$. Suppose $(\ba,\tilde{\bb}) \in \defcone(\tes_{n}(\boldsymbol{1}))$. Let $I$ be the $n \times n$ identity matrix and $\phi$ be the map described in Lemma \ref{vertexcorres}. It is clear that $I$ is a vertex of $\tes_{n}(\boldsymbol{1})$ by Theorem \ref{Tesler}. Let $\bw=\phi(I)$ be the vertex of $Q(\ba,\tilde{\bb})$ corresponding to $I$. Then, by Lemma \ref{vertexcorres}, we have that 
\[ w_{i,j} = -\tilde{b}_{i,j}, \quad \text{for all $1 \le i < j \le n$},\]
  and $w_{i,i}$'s can be determined using hook sum conditions: 
\[w_{1,1}-\tilde{b}_{1,2}-\cdots-\tilde{b}_{1,n}=a_{1},\]
\[(w_{i,i} -\tilde{b}_{i,i+1} - \cdots - \tilde{b}_{i,n})-(-\tilde{b}_{1,i}-\cdots-\tilde{b}_{i-1,i}) = a_{i}, \quad 2 \leq i \leq n. \]
From the above equations, we obtain 
\[ w_{i,i}=\hs_{i}(\tilde{\bb})-\tilde{b}_{i,i}+a_{i} \quad \text{for all $1 \leq i \leq n$.}\]
Since $\bw \in Q(\ba,\tilde{\bb})$, we have inequalities $w_{i,i} \geq -\tilde{b}_{i,i}$ for all $1\leq i \leq n-1$ which implies that $\hs_{i}(\tilde{\bb}) \geq -a_{i}$ for all $1 \leq i \leq n-1$. Therefore, $(\ba,\tilde{\bb})\in A$.

Next we show that $A \subseteq \defcone(\tes_{n}(\boldsymbol{1}))$. Assume that $(\ba,\tilde{\bb}) \in A$ and let $Q:=Q(\ba,\tilde{\bb})$. We will show that $Q$ is a deformation of $\tes_n(\1)$ with deforming vector $(\ba,\tilde{\bb})$. Since $(\ba,\tilde{\bb}) \in A$, we have $(\hs_{1}(\tilde{\bb})+a_{1},\dots,\hs_{n-1}(\tilde{\bb})+a_{n-1},0) \in \R_{\geq 0}^{n}$. Therefore, it follows from Lemma \ref{lem:tight} that
\begin{equation}
T:=\biggl{\{}\bm \in \U(n)~~\biggl{|}
  \begin{array}{c}
    L_{n}\bm= (\hs_{1}(\tilde{\bb})+a_{1},\dots,\hs_{n-1}(\tilde{\bb})+a_{n-1},0)    \\[2mm]
  -P_{n}\bm \leq \0 
  \end{array}  \biggl{\}}
  \label{eq:Tdef}
\end{equation}
  is a non-empty Tesler polytope and the linear inequality description above is tight.
  Let $\bb \in \U(n)$ be the upper triangular matrix obtained from $\tilde{\bb}$ by appending $\displaystyle \sum_{i=1}^{n-1} \tilde{b}_{i,n} -a_n$ as its $(n,n)$th-entry. Then 
  \[ L_n \bb = \Hs(\bb) = (\hs_1(\tilde{\bb}), \dots, \hs_{n-1}(\tilde{\bb}), -a_n) \text{ and } -P_n \bb = - \tilde{\bb}.\]
  It immediately follows that $Q = T - \bb$, is a translation of the Tesler polytope $T$. By Theorem \ref{main}, the polytope $T$ is a deformation of $\tes_n(\1)$, thus so is $Q$ by Remark \ref{rmk:translation}. Moreover, one sees that the linear inequality description \eqref{eq:Qdef} for $Q$ is precisely obtained from the linear inequality description \eqref{eq:Tdef} for $T$ by translating $\bm$ to $\bm - \bb$. Therefore, the tightness of \eqref{eq:Qdef} follows from the tightness of \eqref{eq:Tdef}.
  Hence, we conclude that $(\ba,\tilde{\bb}) \in \defcone(\tes_{n}(\boldsymbol{1}))$, completing the proof of the first assertion of the theorem.

  Finally, as we have shown above, for any $(\ba,\tilde{\bb}) \in A$, the polytope $Q(\ba,\tilde{\bb})$ is a translation of a Tesler polytope. Hence, the second statement of the theorem follows.
\end{proof}

We now proceed to prove Theorem \ref{thm:gen}. 
For all the discussion in the rest of this section, we assume $\ba_0 \in \R_{>0}^n.$  
As we mentioned in the introduction each inequality in the description \eqref{eq:tesdefcone} for $\defcone(\tes_n(\ba_0))$ defines a facet, but we have more than that:  
\begin{lem}\label{lem:faces} For each $I \subseteq [n-1]$, let 
\[ F_I := \left\{(\ba,\tilde{\bb})\in \R^{n}\times \tU(n)~~
\left|~~
\begin{array}{c} \hs_{i}(\tilde{\bb}) \ge -a_{i} \text{ for all $i \in I$}\\ 
  \hs_{i}(\tilde{\bb}) = -a_{i} \text{ for all $i \in [n-1]\setminus I$}
\end{array}
\right.
\right\}\]
be a face of $\defcone(\tes_n(\ba_0))$. Then $I \mapsto F_I$ is an inclusion-preserving bijection from subsets $I$ of $[n-1]$ and faces of $\defcone(\tes_n(\ba_0))$.

Furthermore, for any $\ba \in \R_{\ge 0}^n$, we have that $\cI(\ba) = I$ if and only if $(\ba, \0)$ is in the relative interior of $F_I$.
\end{lem}

\begin{proof}
It follows from the definition of $F_I$ that $I \subseteq J$ if and only if $F_I \subseteq F_J$. It is also clearly, every face of $\defcone(\tes_n(\ba_0))$ arises as an $F_I$ for some $I$. So in order to prove the first conclusion, it is left to show that if $I \subsetneq J$, then $F_I \neq F_J$. 
However, notice that if $\cI(\ba) = I$, then $(\ba,\0)$ belongs to the set
\begin{equation}\label{eq:int} \left\{(\ba,\tilde{\bb})\in \R^{n}\times \tU(n)~~
\left|~~
\begin{array}{c} \hs_{i}(\tilde{\bb}) > -a_{i} \text{ for all $i \in I$}\\ 
  \hs_{i}(\tilde{\bb}) = -a_{i} \text{ for all $i \in [n-1]\setminus I$}
\end{array}
\right.
\right\},
\end{equation}
which is a subset of $F_I$ but clearly has no intersection with $F_J$ if $I \subsetneq J$. Therefore, we conclude that $I \mapsto F_I$ is indeed an inclusion-preserving bijection from subsets $I$ of $[n-1]$ and faces of $\defcone(\tes_n(\ba_0))$.

Next, we observe that the above discussion implies that \eqref{eq:int} actually is the relataive interior of $F_I$. Then the second conslusion follows.
\end{proof}

\begin{proof}[Proof of Theorem \ref{thm:gen}]
  By Lemma \ref{lem:tight}, the Tesler polytope $\tes_n(\ba)$ is a deformation of $\tes_n(\ba_0)$ with deforming vector $(\ba, \0)$. Hence, 
  the conclusion follows from 
  Lemma \ref{lem:faces} and Corollary \ref{cor:type}. 
\end{proof}

\section{Flow polytopes that are deformations of Tesler polytopes} \label{flowsection}

The goal of this subsection is to prove Theorem \ref{charthm} which is a result about flow polytopes. We start by giving a formal definition of flow polytopes on the directed complete graph $\vec{K}_{n+1}$. 
Recall that $\vec{K}_{n+1}$ is the directed complete graph on the set of vertices $[n+1]$ with edge set $E_{n+1}$. A \emph{flow} on $\vec{K}_{n+1}$ is a function $f:E_{n+1} \rightarrow \R_{\geq 0}$ that assigns a non-negative real number to each edge of $\vec{K}_{n+1}$. Given a flow $f$, the \emph{net flow} of $f$ on vertex $i$ is defined to be \[f^{i}:=\displaystyle\sum_{j>i}f(i,j)-\sum_{j<i}f(j,i), \qquad \text{ for all }1\leq i \leq n+1.\] 
Note that $f^{n+1}=-\sum_{1 \leq i \leq n}f^{i}$. Therefore, we often neglect the net flow on the vertex $n+1$, and only discuss net flows on other vertices.

\begin{defn}\label{defn:flow} 
For $\ba=(a_{1},\dots,a_{n}) \in \R^{n}$, 
the \emph{flow polytope on $\vec{K}_{n+1}$} with net flow $\ba$ is defined to be 
\[\flow_{n}(\ba)=\{f:E_{n+1} \rightarrow \R_{\geq 0}~~|~~f^{i}=a_{i} \text{ for all }1\leq i \leq n\}.\]
\end{defn}

Since $\U(n)=\R^{E_{n+1}}$, any flow $f$ on $\vec{K}_{n+1}$ can be considered as an element in $\U(n)$. 
Therefore, we have a natural correspondence between elements in $\U(n)$ and flows on $\vec{K}_{n+1}.$

\begin{ex} 
  Let $\bm \in \U(3)$ be the same upper triangular matrix as in Example \ref{ex:matrix}. 
  The picture below shows $\bm$ together with its corresponding flow $f$ on $\vec{K}_4$. Note that on the right, the number on each edge $e$ is the number $f(e)$ assigned to $e,$ and the number below each vertex $i$ is the net flow $f^i$ of $f$ on $i.$

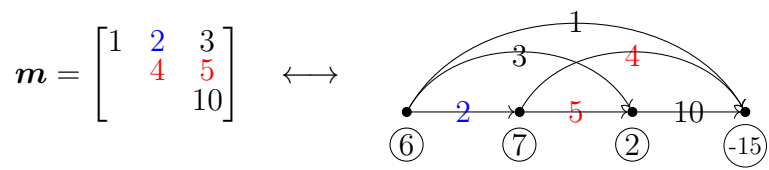
\begin{figure}[h] 

\centering
$\bm=\begin{bmatrix}
1 & {\color{blue}2} & 3 \\
 & {\color{red}4} & {\color{red}5} \\
 & &10 \\
\end{bmatrix} \quad \longleftrightarrow \quad 
\raisebox{-3em}{\begin{tikzpicture}[scale=1.5]
\tikzset{vertex/.style = {shape=circle,fill, draw, inner sep=2pt, scale=0.6}}

\node[vertex] (a) at (0,0) [label={below:$\circled{6}$}]{};
\node[vertex] (b) at (1,0) [label={below:$\circled{7}$}]{};
\node[vertex] (c) at (2,0) [label={below:$\circled{2}$}]{};
\node[vertex] (d) at (3,0) {};

\node[scale=0.8] at (3,-0.3) {$\circled{-15}$};

\node at (0.5,0) {{\color{blue}2}};
\node at (1.5,0) {{\color{red}5}};
\node at (2.5,0) {10};
\node at (1,0.5) {3};
\node at (2,0.5) {{\color{red}4}};
\node at (1.5,0.8) {1};

\draw[->] (a) to (b) ;
\draw[->] (b) to (c) ;
\draw[->] (a) [bend left=60] to node[auto] {} (c);
\draw[->] (a) [bend left=60] to node[auto] {} (d);
\draw[->] (b) [bend left=60] to node[auto] {} (d);
\draw[->] (c) to node[auto] {} (d) ;
The numbers below the vertices indicate the net flows on the corresponding vertices and the numbers above the edges indicate the weight on the corresponding edges.

\end{tikzpicture}}$

\end{figure}
As discussed in Example \ref{ex:matrix}, we have $\hs(\bm)=(6,7,2)$. Thus, we have $f^{i}=\hs_{i}(\bm)$ for $i=1, 2, 3$. 
\end{ex}

It is not hard to see that in general if $\bm \in \U(n)$ is corresponding to the flow $f$ on $\vec{K}_{n+1}$, then $f^{i}=\hs_{i}(\bm)$ for $1 \leq i \leq n$. 
  Therefore, in the matrix notation, we can write the flow polytope with net flow $\ba$ as:
\begin{equation}\label{eq:defnflow}
 \flow_{n}(\ba)=\{\bm \in \U(n)~~|~~L_{n}\bm=\ba \text{  and  } -\bm \leq \boldsymbol{0}\}.
\end{equation} 

In order to be consistent with the bold fonts we have been using for elements in $\U(n),$ below we will use $\beff$ instead of $f$ to denote a flow in $\flow_n(\ba).$
 
One sees that if $\ba \in \R_{\ge 0}^n$, then $\tes_n(\ba)= \flow_n(\ba).$ However, since for flow polytopes, we are allowed to use $\ba$ with negative entries, the family of flow polytopes contains more polytopes than the family of Tesler polytopes.
For convenience, for the rest of the paper, we fix $\ba_0$ as an arbitrary element in $\R_{>0}^n$.
By Theorem \ref{thm:tesdefcone}, any deformation of $\tes_n(\ba_{0})$ is a Tesler polytope up to a translation. Therefore, it is meaningful to ask when a flow polytope is a deformation of $\tes_n(\ba_{0})$. Theorem \ref{charthm} provides an answer to this question.  

\begin{rmk}\label{rmk:redundant}
  Description \eqref{eq:defnflow} is different from
    \begin{equation}
\{\bm \in \U(n)~~|~~L_{n}\bm=\ba \text{  and  } -P_{n}\bm \leq \boldsymbol{0}\}.
\label{eq:flowmatrix0},
\end{equation}
because \eqref{eq:defnflow} has one additional inequality $m_{n,n} \ge 0$. Only when this inequality is redundant in the description \eqref{eq:defnflow} for $\flow_n(\ba)$, we can describe $\flow_n(\ba)$ by \eqref{eq:flowmatrix0}, and then it is possible for $\flow_n(\ba)$ to be a deformation of $\tes_n(\ba_{0}).$ If $m_{n,n} \ge 0$ is not a redundant inequality in \eqref{eq:defnflow}, then $\flow_n(\ba)$ cannot be described by \eqref{eq:flowmatrix0}, and thus cannot be a deformation of $\tes_n(\ba_{0})$.
\end{rmk}

Before we can prove Theorem \ref{charthm}, we need to clarify and define conditions given in the statement of the theorem. First, one observes that
\[\flow_{n}(\ba) \text{ is non-empty if and only if }
\sum_{i =1}^k a_{i} \geq 0 \text{ for all }1\leq k \leq n.\]
For convenience, we let $A_{n}$ be the set of all $\ba \in \R^{n}$ that satisfies the above condition.

\begin{defn}\label{defn:critical} 
  Let $\ba = (a_1, \dots, a_n) \in A_n$. For each $1 \le i \le n-1$, if $a_i > 0$ and $a_i+a_{i+1} = 0$, we say the positive entry $a_i$ is \emph{voided}.

We say $l$ ($1 \le l \le n-1$) is the \emph{critical position} of $\ba$ if $a_l$ is the first positive entry in $\ba$ that is not voided. If such a positive entry does not exist, we say $l=n$ is the \emph{critical position} of $\ba$.
\end{defn}
The following lemma gives a characterization for the entries that appear before the critical position of a point $\ba \in A_n$. 

\begin{lem}\label{lem:charcpp}
  Let $\ba = (a_1, \dots, a_n) \in A_n$. Suppose $l$ is the critical position of $\ba.$ Then for each $1 \le i < l$, exactly one of the following three situations happens:
  \begin{enumerate}
    \item $a_i=0$.
    \item $a_i$ is a voided positive entry.
    \item $a_i < 0$ and $a_{i-1}$ is a voided positive entry.
  \end{enumerate}
  Moreover, if $1 \le l \le n-1$, then $\displaystyle \sum_{i=1}^{l-1} a_i = 0,$ and hence $a_l + a_{l+1} > 0$. 

\end{lem}

\begin{proof}
  The first part (which charaterizes entries $a_i$ for $i < l$) follows directly from the definition of the critical position, and the second conclusion follows immediately from the first part. 
\end{proof}

\begin{lem} \label{lem:beforecp}
  Let $\ba=(a_{1},\dots,a_{n})\in A_{n}$ where $l$ is the critical position of $\ba$ and let $a_{i_{1}},a_{i_{2}},\dots,a_{i_{m}}$ be all the voided positive entries before the critical position where $1\le i_1 < i_2 < \dots < i_m < l$. Then all flows $\beff=(f_{i,j}) \in \flow_{n}(\ba)$ have \emph{fixed} entries in their first $l-1$ rows as described below: 
\begin{enumerate}
  \item For each $1 \leq p \leq m$, the $({i_{p},i_{p}+1})$-th entry $f_{i_{p},i_{p}+1}$ of $\beff$ is always $a_{i_{p}}$.
  \item The remaining entries in the first $l-1$ rows of $\beff$ are zeros.
\end{enumerate}
\end{lem}

\begin{proof} Using Lemma \ref{lem:charcpp} together with the hypothesis of this lemma, we obtain that
  \[ a_i = \begin{cases} -a_{i+1}, \quad & \text{if $i = i_p$ for some $1 \le p \le m$;} \\
      0, \quad & \text{if $1 \le i < l$ and $i \neq i_p$ or $i_p+1$ for all $1 \le p \le m$.}
  \end{cases}\]
Since the first $i_{1}-1$ entries of $\ba$ are zeros and all the entries in $\beff$ are non-negative, we conclude that the first $i_{1}-1$ rows of $\beff$ have to be zero rows. In particular,
  \[ f_{i, i_1} = 0, \quad \text{for all $1 \le i \le i_1-1$}.\]
  Hence, the $i_{1}$-th and the $(i_1+1)$-st hook sums of $\beff$ are 
  \begin{align}
    \hs_{i_{1}}(\beff)=& \ f_{i_{1},i_{1}}+\dots+f_{i_{1},n}-(f_{1,i_{1}}+\cdots+f_{i_{1}-1,i_{1}})=f_{i_{1},i_{1}}+\dots+f_{i_{1},n}=a_{i_{1}},   \label{eq:hk1} \\
    \hs_{i_{1}+1}(\beff)=& \ f_{i_{1}+1,i_{1}+1}+\dots+f_{i_{1}+1,n}-(f_{1,i_{1}+1}+\cdots+f_{i_{1},i_{1}+1})  \label{eq:hk2} \\
=& \ f_{i_{1}+1,i_{1}+1}+\dots+f_{i_{1}+1,n}-f_{i_{1},i_{1}+1}= a_{i_{1}+1}=-a_{i_{1}}. \nonumber
\end{align}
Using the fact that $\beff$ has non-negative entries again, one sees that \eqref{eq:hk1} implies $f_{i_{1},i_{1}+1} \leq a_{i_{1}}$ and \eqref{eq:hk2} implies $f_{i_{1},i_{1}+1} \geq a_{i_{1}}$.
Thus, $f_{i_{1},i_{1}+1}=a_{i_{1}}$. This forces the remaining entries in the $i_{1}$-th and $(i_{1}+1)$-st rows are zeros. We have shown that the first $i_1+1$ rows of $\beff$ are given as described by the lemma. By iterating similar arguments, one can deduce that all the first $l-1$ rows are as described. 
\end{proof}
The following two propositions are the key results that will be used in our proof for Theorem \ref{charthm}.

\begin{prop}\label{prop:flowtrans}
  Let $\ba = (a_1, \dots, a_n) \in A_n$ and suppose $l$ is the critical position of $\ba$. 
  \begin{enumerate}
    \item\label{itm:flowtrans1} If $1 \le l < n,$ then there exists $\hat{\ba}= (\hat{a}_1, \dots, \hat{a}_n) \in A_n$ satisfying
      \[ \hat{a}_1=\hat{a}_2 = \cdots = \hat{a}_{l-1}=0, \ \hat{a}_l > 0, \ \hat{a}_{l+1} \ge 0, \text{ and $\hat{a}_i = a_i$ for $l+2 \le i \le n$.}  \]
  such that $\flow_{n}(\ba)$ is a translation of $\flow_{n}(\hat{\ba})$. 
\item\label{itm:flowtrans2} If $l = n,$ then $\flow_{n}(\ba)$ is a point. 
  \end{enumerate}
\end{prop}

\begin{prop} \label{prop:nondeform}
Let $n \geq 3$ and $\ba=(a_{1},a_{2},\cdots,a_{n}) \in A_{n}$. 
  Suppose for some integer $1 \le l \le n-1$, we have $a_1=a_2 = \cdots = a_{l-1}=0$, $a_l > 0$, and $a_{l+1} \ge 0$. (Note that this implies that $l$ is the critical position of $\ba.$) 
Assume further that $a_{m}$ is the first negative entry of $\ba$, where $l+2 \le m \le n.$
Then $\flow_n(\ba)$ is not a deformation of $\tes_n(\ba_{0})$.
\end{prop}

Before we prove each of these two propositions, we need a preliminary lemma. 
Recall that $\be_k$ is the $k$-th standard basis vector of $\mathbb{R}^n$.
For convenience, for every $1 \le i \le j \le n,$ we denote by $\be_{i,j}$ to be the unique upper triangular matrix in $\U(n)$ whose $(i,j)$-th entry is its only nonzero entry and has value $1$. So $\{ \be_{i,j}\}$ is the standard basis for $\U(n)$.

\begin{lem}\label{lem:translation}
  Let $\ba = (a_1, \dots, a_n) \in A_n$ and $c > 0$. Suppose for a fixed pair of indices $1 \le i \le j\le n,$ we have that every flow $\beff=(f_{i,j}) \in \flow_n(\ba)$ satisfies $f_{i,j} \ge c.$ Then
  \[ \flow_n(\ba) - c \be_{i,j} = \begin{cases}
    \flow_n(\ba - c\be_i + c\be_j), & \quad \text{if $i < j$}, \\
  \flow_n(\ba - c\be_i), & \quad \text{if $i=j$}.
\end{cases}\]
\end{lem}

\begin{proof} We only prove the case when $i <j$, as the case when $i=j$ can be proved similarly. Let $\beff \in \U(n)$ and $\bg = \beff - c\be_{i,j}.$
  First, since $\Hs(\be_{i,j})= \be_i -\be_j$, we have that $\Hs(\beff) = \ba$ if and only only if $\Hs(\bg) = \ba - c\be_i + c\be_j.$ 
  Next, if $\beff \in \flow_n(\ba)$, because $f_{i,j} \ge c,$ we see that all entries of $\bg$ are non-negative. Conversely, if $\bg \in \flow_n(\ba - c\be_i + c\be_j)$, since $c > 0,$ we must have that all entries of $\beff=\bg + c\be_{i,j}$ are non-negative. 
  Therefore, we conclude that $\beff \in \flow_n(\ba)$ if and only if $\bg \in \flow_n(\ba - c\be_i + c\be_j).$ Then the conclusion of the lemma follows.
\end{proof}

\begin{proof}[Proof of Proposition \ref{prop:flowtrans}] 
  Suppose $l=n$. Then by Lemma \ref{lem:beforecp}, the first $n-1$ rows of all the flows in $\flow_n(\ba)$ are fixed. However, the $n$-th hook sum condition makes the only entry in the $n$-th row to be fixed as well. Hence, \eqref{itm:flowtrans2} follows.

  Suppose $1 \le l \le n-1$ and assume $a_{i_{1}},a_{i_{2}},\dots,a_{i_{m}}$ are all the voided positive entries before the critical position where $1\le i_1 < i_2 < \dots < i_m < l$. It follows from Lemma \ref{lem:beforecp} and Lemma \ref{lem:translation} that
  \[ \flow_{n}(\ba)- \sum_{p=1}^m a_{i_p} \be_{i_p, i_p+1} =\flow_{n}(0,\dots,0,a_{l},a_{l+1},\dots,a_{n}).\]
  If $a_{l+1}\geq 0$, then we see that $\ba':=(0,\dots,0,a_{l},a_{l+1},\dots,a_{n})$ is a desired choice for $\hat{a}$. 

  Suppose $a_{l+1}<0$. We now consider the flow polytope $\flow_{n}(\ba')$. Clearly, $l$ is still the critical position of the vector $\ba'$. Let $\bg \in \flow_{n}(\ba').$ It is clear (or follows from Lemma \ref{lem:beforecp}) that the first $l-1$ rows of $\bg$ are zero rows.
  Thus, the $(l+1)$-th hook sum of $\bg=(g_{i,j})$ is: \[\hs_{l+1}(\bg)=g_{l+1,l+1}+\cdots+g_{l+1,n}-g_{l,l+1}=a_{l+1} \]
Since $g_{l+1,l+1}, \cdots, g_{l+1,n}$ are all non-negative, we get that $g_{l,l+1} \ge - a_{l+1} > 0$. Now applying Lemma \ref{lem:translation} with $c = -a_{l+1}$, we get
\[ \flow_n(\ba') - c \be_{l,l+1} = \flow_n( \ba' - c\be_l + c\be_{l+1}).\]
where
\[ \ba' - c\be_l + c\be_{l+1} = (0, \dots, 0, a_l + a_{l+1}, 0, a_{l+2}, \dots, a_n).\]
However, by the last conclusion of Lemma \ref{lem:charcpp}, we have $a_l + a_{l+1} > 0.$ Hence, $\ba' - c\be_l + c\be_{l+1}$ is a desired choice for $\hat{\ba}$ for this case. Therefore, \eqref{itm:flowtrans1} holds. 
\end{proof}

The next lemma shows the existence of a certain flow which will be used in the proof of Proposition \ref{prop:nondeform}. 

\begin{lem} \label{flowlem}
Assume the hypothesis of Proposition \ref{prop:nondeform}. If $m \neq n$, i.e., $l +2 \le m < n$, then there exists a flow $\beff=(f_{i,j})$ in $\flow_{n}(\ba)$ satisfying the following conditions:

\begin{enumerate}[label=(\roman*)]
  \item\label{itm:pos} $f_{l,m} > 0$.
  \item\label{itm:-am} $f_{l,m} + f_{l+1,m} + \cdots + f_{m-1,m} = -a_m.$
  \item\label{itm:zerorow} 
    The $m$th row of $\beff$ is a zero row.
  \item\label{itm:zerorows} The first $l-1$ rows of $\beff$ are zero rows. 
\end{enumerate}
\end{lem}

\begin{proof}
  We prove the lemma by constructing such a flow $\beff$. Because the first $l-1$ entries of $\ba$ are all zero and $a_m$ is the first negative entry, we have that \[\sum_{p=1}^{l-1} a_p =0, \ \sum_{p=1}^l a_p, \ \sum_{p=1}^{l+1} a_p,  \dots, \sum_{p=1}^{m-1} a_p\] is a weakly increasing sequence. Since $\ba \in A_{n}$, one sees that the last entry in the above sequence is at least $-a_m$, which is strictly greater than the first entry in the sequence.
  Therefore, there exists (a unique) $k \in \{l, l+1, \dots, m-1\}$ such that 
\[ \sum_{p=1}^{k-1} a_p < -a_{m} \le \sum_{p=1}^k a_p.\]
Let 
\[ c_i = \begin{cases} a_i, & \text{if $l \le i \le k-1$} \\
    -a_m - \sum_{p=1}^{k-1} a_p, & \text{if $i=k$} \\
    0, & \text{if $k+1 \le i \le m-1$}.
\end{cases},
\]
and then let 
\[d_i = a_i - c_i, \text{ for $l \le i \le m-1$}.\]
One can verify that all $c_i$'s and $d_i$'s are non-negative, and moreover,
\begin{equation}
  c_l > 0 \quad \text{and} \quad c_l + c_{l+1} + \cdots c_{m-1} = -a_m.
  \label{eq:ci}
\end{equation}
Now we are ready to construct a desired flow $\beff$. We define $\beff=(f_{i,j}) \in \U(n)$ as 
    \[ f_{i,j} = \begin{cases} c_i, & \text{if $l \le i \le m-1$ and $j=m$} \\
	d_i, & \text{if $l \le i \le m-1$ and $j=m+1$} \\
	\sum_{p=1}^i a_p, &\text{if $m+1 \le i \le n-1$ and $j=i+1$} \\
	\sum_{p=1}^n a_p, & \text{if $i=j=n$} \\
	0, & \text{otherwise}.
    \end{cases}\]
It is straightforward to check that $\beff \in \flow_{n}(\ba)$ and the desired properties of $\beff$ follow from the definition of $\beff$ and \eqref{eq:ci}.
\end{proof}

\begin{proof}[Proof of Proposition \ref{prop:nondeform}]

  Suppose $m=n.$ Then $a_1, a_2, \dots, a_{n-1} \ge 0$ and $a_n < 0.$ Consider $\bg = (g_{i,j}) \in \U(n)$ defined by $g_{i,i} = a_i$ for each $i$ and $g_{i,j} = 0$ for each $1 \le i <j \le n.$ Clearly $\bg \notin \flow_n(\ba)$ as $g_{n,n} = a_n < 0.$ However, $\bg$ is a point in the polytope defined by \eqref{eq:flowmatrix0}. This means if we remove the inequality $m_{n,n} \ge 0$ from the inequality description \eqref{eq:defnflow} for $\flow_n(\ba)$, we obtain a different polytope. Hence, the inequality $m_{n,n} \ge 0$ is not redundant in \eqref{eq:defnflow}. Hence, by Remark \ref{rmk:redundant}, we conclude that $\flow_n(\ba)$ is not a deformation of $\tes_n(\ba_{0})$.

  Suppose $m \neq n.$ By Remark \ref{rmk:redundant}, we only need to consider the situation when \eqref{eq:flowmatrix0} is a linear inequality description for $\flow_n(\ba)$. Note that the description is not necessarily tight. However, it is clear that there exists $\tilde{\bb}=(\tilde{b}_{i,j}) \in \widetilde{\U}(n)$ such that 
\begin{equation}\label{flowtight}
  \flow_{n}(\ba)=Q(\ba, \tilde{\bb}) = \{\bm \in \U(n)~~|~~L_{n}\bm=\ba \text{  and  } -P_{n}\bm \leq \tilde{\bb}\}.
\end{equation}
is a tight linear inequality description for $\flow_n(\ba).$ By Remark \ref{rmk:tight}, one sees that it is enough to show that $(\ba, \tilde{\bb})$ is not in the deformation cone of $\tes_n(\ba_{0})$, which by Theorem \ref{thm:tesdefcone} is equivalent to that $\hs_{i}(\tilde{\bb}) < -a_{i}$ for some $1 \le i \le n-1.$ We will show that $\hs_{m}(\tilde{\bb}) < -a_{m}$. Note that by the definition of tightness and the assumption that $\flow_n(\ba)$ is defined by \eqref{eq:flowmatrix0}, the entries in $\tilde{\bb}$ satisfy:
\begin{equation}
  0 \le -\tilde{b}_{i,j} = \min_{\bm \in \flow_n(\ba)} m_{i,j}.
  \label{eq:entrytildeb}
\end{equation}

Let $\beff=(f_{i,j})$ be the flow in $\flow_{n}(\ba)$ assumed by Lemma \ref{flowlem}. 
Then condition \eqref{eq:entrytildeb} implies that  
\[ -\tilde{b}_{i,m} \le f_{i,m} \text{ for $l \le i \le m-1$}.\]
Using condition \eqref{eq:entrytildeb} together with Lemma \ref{flowlem}/\ref{itm:zerorow}\ref{itm:zerorows}, we conclude that
\[ \tilde{b}_{m,m}=\tilde{b}_{m,m+1}=\cdots=\tilde{b}_{m,n}=0 \text{ and } \tilde{b}_{1,m} =\tilde{b}_{2,m}=\cdots =\tilde{b}_{l-1,m} = 0.\]
Next, by Lemma \ref{flowlem}/\ref{itm:pos}, we may choose a number $\epsilon$ such that $0 < \epsilon < f_{l,m}$. Then we construct $\bg=(g_{i,j})\in \U(n)$ from $\beff$ by letting 
\[ g_{l,m}=f_{l,m}-\epsilon, \ g_{l+1,m}=f_{l+1,m}+\epsilon, \ g_{l,l+1}=f_{l,l+1}+\epsilon,\] and keeping all the other entries. 
One can check that $\bg$ is another flow in $\flow_{n}(\ba)$. Applying \eqref{eq:entrytildeb} again for $(i,j)=(l,m)$, we get
\[ -\tilde{b}_{l,m} \le g_{l,m} = f_{l,m}-\epsilon.\]
Therefore,
\begin{align*}
\hs_{m}(\tilde{\bb}) =& \sum_{j=m}^n \tilde{b}_{m,j} - \sum_{i=1}^{m-1} \tilde{b}_{i,m} = \sum_{j=m}^n \tilde{b}_{m,j} -\sum_{i=1}^{l-1} \tilde{b}_{i,m}- \sum_{i=l}^{m-1} \tilde{b}_{i,m} = 0 - 0 - \sum_{i=l}^{m-1} \tilde{b}_{i,m} \\
\le&  (f_{l,m}-\epsilon) +  f_{l+1,m} + \cdots + f_{m-1,m} < f_{1,m} + f_{2,m} + \cdots + f_{m-1,m} = -a_m,
\end{align*}
where the last equality follows from Lemma \ref{flowlem}/\ref{itm:-am}. This completes our proof.
\end{proof}

\begin{proof}[Proof of Theorem \ref{charthm}]

  By Proposition \ref{prop:flowtrans}/\eqref{itm:flowtrans2}, if $l=n,$ then $\flow_n(\ba)$ is a point, which clearly is a deformation of $\tes_n(\ba_{0})$. Therefore, it is left to prove that if $1 \le l \le n-1,$ then
  \begin{equation}
    \text{  $\flow_n(\ba)$ is a deformation of $\tes_{n}(\boldsymbol{1})$}  \quad \Longleftrightarrow \quad  a_{l+2},a_{l+3},\dots,a_{n} \geq 0.
    \label{eq:iff}
  \end{equation}
  However, by Proposition \ref{prop:flowtrans} and Remark \ref{rmk:translation}, one sees that it is enough to prove \eqref{eq:iff} with the assumption that $\ba=(a_1, \dots, a_n)$ satisfies 
\[ a_1=a_2 = \cdots = a_{l-1}=0, \ a_l > 0, \ \text{ and } a_{l+1} \ge 0.\]
Then the forward direction of \eqref{eq:iff} immediately follows Proposition \ref{prop:nondeform}. 
If $a_{l+2},a_{l+3},\dots,a_{n} \geq 0,$ then $\ba \in \R^{n}_{\geq 0}$ and $\flow_{n}(\ba)=\tes_{n}(\ba)$, which by Theorem \ref{main} is a deformation of $\tes_{n}(\boldsymbol{1})$. Hence, the backward direction of \eqref{eq:iff} holds.
\end{proof}

\bibliographystyle{abbrv}
\bibliography{ref}

\end{document}